\documentclass{amsart}
\usepackage{euscript,graphicx,epstopdf,amscd,amsgen,amsfonts,amssymb,latexsym,amsmath,amsthm,graphicx,mathrsfs,times,color,overpic,bbm}
\usepackage{enumitem}

\usepackage{xcolor}

\definecolor{forestgreen}{rgb}{0.0, 0.27, 0.13}

\newtheorem{theorem}{Theorem}[section]

\newtheorem{lemma}[theorem]{Lemma}
\newtheorem{corollary}[theorem]{Corollary}
\newtheorem{proposition}[theorem]{Proposition}
\newtheorem{claim}[theorem]{Claim}

\theoremstyle{definition}
\newtheorem{example}[theorem]{Example}
\newtheorem{remark}[theorem]{Remark}

\def\bN{\mathbb{N}}
\def\bZ{\mathbb{Z}}
\def\bR{\mathbb{R}}

\def\bS{\mathbb{S}}
\def\bP{\mathbb{P}}
\def\bM{\mathbb{M}}

\def\cF{\mathcal{F}}

\def\cI{\mathcal{I}}

\def\bi{\mathbf{i}}
\def\bj{\mathbf{j}}

\def\bi{{\mathbf i}}
\def\bj{{\mathbf j}}

\DeclareMathOperator{\prob}{Prob}
\DeclareMathOperator{\diam}{diam}

\DeclareMathOperator{\supp}{supp}
\DeclareMathOperator{\card}{card}

\DeclareMathOperator{\GL}{GL}
\newcommand{\eqdef}{\stackrel{\scriptscriptstyle\rm def}{=}}

\numberwithin{equation}{section}

\DeclareMathSymbol{\emptyset}{\mathord}{AMSb}{"3F}
\title[Synchronization rates and limit laws]{Synchronization rates and limit laws\\for random dynamical systems}

\author[K. Gelfert]{Katrin Gelfert}
\address{Instituto de Matem\'atica, Universidade Federal do Rio de Janeiro, Cidade Universit\'aria - Ilha do Fund\~ao, Rio de Janeiro 21945-909,  Brazil}
\email{gelfert@im.ufrj.br}

\author[G. Salcedo]{Graccyela Salcedo}
\address{Faculty of Mathematics and Computer Science, Nicolaus Copernicus University, ul. Chopina 12/18,87-100 Toruń, Poland}
\email{gsalcedo@mat.umk.pl}

\begin{document}

\begin{abstract}
We study general random dynamical systems of continuous maps on some compact metric space. Assuming a local contraction condition and uniqueness of the stationary measure, we establish probabilistic limit laws such as the central limit theorem, the strong law of large numbers, and the law of the iterated logarithm. Moreover, we study exponential synchronization and synchronization on average. In the particular case of iterated function systems on $\bS^1$, we analyze synchronization rates and describe their large deviations.  In the case of $C^{1+\beta}$-diffeomorphisms, these deviations on random orbits are obtained from the large deviations of the expected Lyapunov exponent.
\end{abstract}

\begin{thanks}
{This study was financed in part by the FAPERJ-grants 
PRONEX E-26/010.001252/2016 FAPERJ,  % Pronex Marcelo
E-16/2014 INCT/FAPERJ, % lorenzo inctmat
E-26/200.371/2023 CNE/FAPERJ, %  Katrin
CNPq-grant 305327/2022-4,  % Katrin
 (Brazil); 
 the grant 2022/10341-9 São Paulo Research Foundation (FAPESP); and the Center of Excellence “Dynamics, Mathematical Analysis and Artificial Intelligence” at Nicolaus Copernicus University in Toruń. The authors would like to thank Anton Gorodetski for his invaluable comments and contributions regarding Example \ref{exemplo:Anton}.
}
\end{thanks}

\subjclass[2000]{}

\maketitle

%-----------------------------------------
\section{Introduction}\label{secintro}
%-----------------------------------------

Let $(M,D)$ be a compact metric space and consider the space $C^0(M)$ of continuous maps  $f\colon M\to M$ endowed with the supremum distance $D_\infty$\footnote{For $f,g\in C^0(M)$, $D_{\infty}(f,g)=\sup_{x\in M}D(f(x),g(x))$.}. Consider a subset $\cF$ of $C^0(M)$ and index it as $\cF=\{f_i\}_{i\in\cI}$ for some index set $\cI$ (possibly using $\cI=\cF$ itself).  Fix a probability measure $\mu$ on $\cI$. We assume that $\mu$ is non-degenerate, that is, $\cI$ is the support of $\mu$.
We will call $(\cF,\mu)$ a \emph{random dynamical system (RDS)} on $M$. When $\cI$ is finite, $(\cF,\mu)$ is also called an \emph{iterated function system (IFS) with probabilities}.

The \emph{random walk} generated by $\mu$ on $C^0(M)$ is the random
sequence $\bi\mapsto (f_{\bi}^n)_{n\in\bN}$ of elements of $C^0(M)$, defined by
\[
\bi=(i_1,i_2,\ldots)\in \cI^\bN,\,n\in\bN,\quad f_{\bi}^{n}\eqdef f_{i_n}\circ\cdots\circ f_{i_1}.
\]
It can be convenient alongside the random walk-point of view to take a dynamical systems point of view and consider the associated \emph{skew-product}, that is the transformation
\begin{equation}\label{def:skewprod-RD}
	T\colon \cI^{\bN}\times M\to \cI^{\bN}\times M,\quad
	T(\bi,x)\eqdef (\sigma(\bi),f_{i_1}(x)),
\end{equation}
where $\sigma$ is the shift operator on $\cI^{\bN}$ and $i_1$ is the first coordinate of $\bi$. Throughout this paper, we consider the product%
\footnote{The considered sigma-algebra of $\cI$ is the one induced by the Borel sigma-algebra of $(\cF,{D_\infty}|_{\cF})\subset(C^0(M),D_\infty)$. Note that $\mu^\bN$ is uniquely determined by $\mu$ (see, Kolmogorov's extension theorem). If $\cF$ is finite, then the index set can be taken $\cI=\{0,\ldots,N-1\}$, $N=\card\cF$, and $\mu^{\bN}$ a Bernoulli probability measure.}
measure $\mu^{\bN}$ on $\cI^{\bN}$. Note that the measure $\mu^{\bN}$ is $\sigma$-invariant and, by Kolmogorov’s 0-1 law, it is $\sigma$-ergodic.

Given $\bi=(i_1,i_2,\ldots)\in\cI^{\bN}$, $n\in\bN$, and $x\in M$, we define
\begin{align}\label{chain:PropEst}
	X_n^x(\bi)
	\eqdef f_\bi^n(x),
	\quad
	X_0^x(\bi)
	=x,
\end{align}
which can be interpret as the ``trajectory'' corresponding to the realization $\bi$ of the random sequence $(f_{\bi}^n)_{n\in\bN}$ starting at $x$. In other words, $(X_n^x)_{n\ge0}$ is the Markov chain with an initial distribution being the Dirac measure $\delta_x$. A Borel probability measure $\nu$ on $M$ is \emph{$\mu$-stationary} if it is a fixed point of the \emph{Markov operator} $\bM\colon \mbox{Prob}(M)\to\mbox{Prob}(M)$ that acts on the space of Borel probability measures on $M$, $\mbox{Prob}(M)$, and is defined by
\begin{equation}\label{cFast}
	\nu
	= \bM\nu
	\eqdef \int (f_i)_\ast\nu\,d\mu(i),
	\quad\text{ where }\quad
	(f_i)_\ast\nu\eqdef\nu\circ f_i^{-1}
\end{equation}
(to simplify notation, in the above definition we suppress its dependence on $\mu$).
This term is justified by the fact that if $X_0$ is a $\nu$-distributed random variable, then
\[
	X_n
	\eqdef f_{I_n}\circ\cdots\circ f_{I_1}(X_0)
\]
 is a stationary stochastic sequence with values in $M$, where $(I_n)_{n\in\bN}$ is an i.i.d.\footnote{The variables $I_n$ are \emph{identically and independently distributed}, that is, they have all the same distribution and are all mutually independent.} sequence of  random variables distributed on $\cI$ according to $\mu$. Since all maps in $\cF$  act continuously on the compact metric space $M$, this chain has the \emph{weak Feller property} and hence there exists at least one $\mu$-stationary measure. Note that for any $\mu$-stationary probability measure $\nu$, the probability measure $\mu^\bN\otimes\nu$ on $\cI^\bN\times M$ is invariant and ergodic with respect to the skew-product map \eqref{def:skewprod-RD}.

 For systems that are (uniformly) contracting, the seminal work of Hutchison \cite{HUTCH} gave the first fundamental results about the existence of the (topological) attractor and the existence of the (unique) stationary measure, considering however only a finite number of contractions (that is, a contracting IFS with probabilities) on a complete metric space.  One way to extend Hutchinson's results and ideas is to consider larger spaces of contracting maps. In \cite{Lew:93} and \cite{Men:98}, compact metric spaces were considered, however \emph{a priori} still under the assumptions of contraction on average (see property \textbf{(CA)} defined below) and (uniform) contraction, respectively. One of our main motivations to consider the more general setting of RDS that are not necessarily contracting is the currently very intense study of matrix cocycles (see, for example, \cite{DuaKle:17}) and of partially hyperbolic dynamical systems of skew-product type (e.g. \cite{DiaGelRam:17} and references therein). By their nature, such systems usually do not satisfy any contraction, but rather have ``coexisting expanding and contracting regions''.

Let us state our first key hypotheses in what is studied below.
\begin{itemize}[ leftmargin=1cm ]
\item[\textbf{(P)}]  $\cF$ is \emph{proximal} on $(M,D)$ if for every $x,y\in M$ there exist $\bi\in\cI$ and a sequence $(n_k)_{k\in\bN}$ such that
\[
\lim_{k\to\infty} D\left(f_\bi^{n_k}(x),f_\bi^{n_k}(y) \right)=0.
\]
\item[\textbf{(LC)}] $(\cF,\mu)$ has a \emph{local contraction property} on $(M,D)$ if there exists a \emph{local contraction rate} $q\in(0,1)$ so that for every $x\in M$, for almost every $\bi\in\cI$ there exists a neighborhood $B$ of $x$ such that for all $n\in\bN$,
\[
	\diam_D\big(f_\bi^n(B)\big)
	\le q^n,
\]
where $\diam_D B\eqdef\sup_{y,z\in B}D(y,z)$ is the diameter of $B$ with respect to $D$.
\end{itemize}
Note that property \textbf{(P)} is purely a topological one. In particular, \textbf{(P)} is independent of any non-degenerate measure $\mu$ on $\cI$.
Moreover, by compactness, the property \textbf{(P)} is also unaffected by any metric change on $M$ that preserves the topology. This is in contrast to property \textbf{(LC)}, where a change of non-degenerate measures on $\cI$ can impact the contraction rate. Additionally, even though a change of metrics may maintain the convergence of diameters to $0$, but such convergence may be non-exponential.

We explore the following consequences for an RDS having the properties \textbf{(P)} and \textbf{(LC)}:
\begin{enumerate}[ leftmargin=0.7cm ]
\item synchronization (exponentially fast and on average) (see Section \ref{subsec:INT-1}),
\item limit laws for additive functionals of Markov chains, in particular for
\begin{itemize}
\item a general RDS of continuous maps on a compact metric space (see Section \ref{subsec:INT-2}),
\item  an IFS of homeomorphisms and $C^1$ diffeomorphisms of $\bS^1$ (see Section \ref{sec:IFSoncircle});
\end{itemize}
\item large deviations of synchronization rates (and Lyapunov exponents) for an IFS of $C^{1+\beta}$ diffeomorphisms of $\bS^1$ (see Section  \ref{subsec-INT-LD}).
\end{enumerate}

The phenomenon of synchronization observed in RDSs is the one in which random orbits start at different initial conditions and eventually approach each other. Even though the properties \textbf{(LC)} and synchronization have comparable flavors, to the best of our knowledge, the latter property does not imply the former one.
One further motivation for studying the property \textbf{(LC)} is because it enables us to establish a number of limit laws such as, for example, the central limit theorem (CLT). Below we will provide details on previous work toward items (1)--(3) listed above, in particular about \textbf{(LC)}.

%------------------------------------------------------------------------------------------------------
\subsection{Synchronization: exponentially fast and on average}\label{subsec:INT-1}
%------------------------------------------------------------------------------------------------------

\begin{itemize}[ leftmargin=0.7cm ]
\item[\textbf{(S)}] $(\cF,\mu)$ is \emph{synchronizing} on $(M,D)$ if for every $x,y\in M$ and $\mu^{\bN}$-almost every $\bi\in\cI$,
\[
	\lim_{n\to\infty}D(X_n^x(\bi),X_n^y(\bi))=0.
\]
\end{itemize}

One of the first results towards synchronization was obtained by Furstenberg \cite{HFur:63} and Ledrappier \cite{Led:86} for projectivized linear dynamics and later extended by Baxendale \cite{Bax:89} and Crauel \cite{Cra:90} to the random iteration of diffeomorphisms on compact manifolds. All such results exclude \emph{a priori} the ``degenerate case'' assuming the absence of a common invariant measure (this is hypothesis \textbf{(H)} stated below) and prove the existence of a stationary measure with negative ``volume contraction''. For one-dimensional dynamics, this implies local contraction and synchronization. Assuming additionally minimality, Antonov \cite{Antonov:1984} was able to extend such results to circle \emph{homeomorphisms}. Assuming \textbf{(H)}, Malicet establishes a local (exponential) contraction property for RDS of circle homeomorphisms  (see \cite[Theorem A]{Mal:17} and Remark \ref{reminterplay} for further comments). Matias \cite{Matias:2022} extends \cite{Mal:17} to the Markovian case.

Let us also mention that in the context of Markovian random iterations of finitely many maps on compact subsets of a finite-dimensional Euclidean space and also assuming a purely topological condition on the maps, called \emph{splitting condition}, D\'iaz and Matias \cite[Theorem 2]{DiaMat:18} show an exponentially fast contraction of the whole space under the action of the random iterations. Note that on the circle $\bS^1$, this condition is never satisfied. Homburg \cite{Hom:18} proves that for any compact manifold $M$ there is a $C^1$-open set of IFS of $C^2$ diffeomorphisms on $M$ which is synchronizing and has a unique stationary measure.

We first consider the general setting of a RDS.

\begin{proposition}[Exponential synchronization]\label{prop:sync-exp}
	Assume that $(\cF,\mu)$ is an RDS of continuous maps on a compact metric space $(M,D)$ having the properties \textbf{(P)} and \textbf{(LC)}. Let $q\in(0,1)$ be an associated local contraction rate. Then  $(\cF,\mu)$ is synchronizing \textbf{(S)} on $(M,D)$ and has a unique $\mu$-stationary probability measure.
 Moreover, for every $x,y\in M$ and $\mu^{\bN}$-almost every $\bi$ there is $C>0$ so that
\begin{align}\label{prop1.1:syncexp}
    D(X^x_n(\bi),X^y_n(\bi))
	\le Cq^n\quad\text{ for all }n\in\bN.
\end{align}
	
\end{proposition}
In order to obtain \textbf{(S)}, both hypotheses \textbf{(P)} and \textbf{(LC)} are indeed necessary, as it is illustrated by the following example.

\begin{example}\label{exemplo:Anton}
    Consider $f_1,f_2$ two orientation preserving diffeomorphisms on the circle $\bS^1=[0,1)$. Assume that the maps $f_1$ and $f_2$ each have exactly four fixed points, two attracting and two repelling ones. Let $d$ be the usual metric on $\bS^1$. Let $\{0,1/4,1/2,3/4\}$ and $\{1/8,3/8,5/8,7/8\}$ be the fixed points of $f_1$ and $f_2$, respectively. We assume that
\[\begin{split}
	\min\{|f_1'(0)|,|f_1'(1/2)|, |f_2'(1/8)|,|f_2'(5/8)| \}&>1,\\
    	\max\{|f_1'(1/4)|,|f_1'(3/4)|, |f_2'(3/8)|,|f_2'(7/8)| \}&<1.
\end{split}\]	
Now, consider $f_3\colon\bS^1\to\bS^1$ given by $f_3(x)=x+1/2 \mod 1 $. Let $\cI=\{1,2,3\}$ and $\cF=\{f_i\}_{i\in\cI}$. For any non-degenerate probability measure $\mu$ on $\cI$ (that is, satisfying $\mu(\{i\})>0$ for $i\in\cI$), the pair $(\cF,\mu)$ has the property of local contraction \textbf{(LC)}. This is a direct consequence of \cite[Theorem A]{Mal:17} and the fact that $f_1$ and $f_2$ have no common fixed points. Note that for any pair of points $x\in [1/4,3/8]$ and $y\in [3/4,7/8]$ and any sequence $\bi\in\cI$,
    \[
    d(X^x_n(\bi),X^y_n(\bi))=d(f_{\bi}^n(x),f_{\bi}^n(y))\geq \frac{3}{8},
    \]
    for all $n\in \bN$. Therefore, $\cF$ is not proximal and so $(\cF,\mu)$ is not synchronizing.
\end{example}

\begin{corollary}\label{cor:atomic}
    Assume that $(\cF,\mu)$ is an RDS of continuous maps on a compact metric space $(M,D)$ having the properties \textbf{(P)} and \textbf{(LC)}. Let $\nu $ be the $\mu$-stationary probability measure on $M$. If all maps in $\cF$ are injective, then we have the following dichotomy:
    \begin{enumerate}
        \item either $\nu$ is a Dirac measure at a fixed point of $f_i$ for $\mu$-almost every $i$,
        \item or $\nu$ is non-atomic.
    \end{enumerate}
\end{corollary}

Given any pair of points $x,y$, the constant $C$ in the assertion of Proposition \ref{prop:sync-exp} depends on the random walk generated by $\mu$ and it is \emph{a priori} not clear that it can be guaranteed to be $\mu^{\bN}$-integrable. Integrability would immediately imply the following exponential synchronization on average. We prove this property assuming additionally that $(M,D)$ is compact. For further applications, we include the possibility of a change to an equivalent metric.%
\footnote{Two metrics are \emph{(topologically) equivalent} if they generate the same topology.
It is straightforward to check that $D^\alpha$, $\alpha\in(0,1]$, indeed defines a metric on $M$. Moreover, $D^\alpha$ is equivalent to $D$.}

\begin{theorem}\label{teo:sync-on-average}
Assume that $(\cF,\mu)$ is an RDS of continuous maps on a compact metric space $(M,D)$ having the properties \textbf{(P)} and \textbf{(LC)}. Let $q\in(0,1)$ be an associated local contraction rate. Then for any $\alpha\in(0,1]$ there exists $C>0$ such that for all $n\in\bN$
\[
\sup_{x,y \in M}  \sum_{k=0}^{n} \int D^\alpha(X_k^x(\bi),X_k^y(\bi))\,d\mu^\bN(\bi)\leq C\sum_{k=0}^{n}q^{k\alpha}.
\]
\end{theorem}

The following property is slightly stronger than the one asserted above.
\begin{itemize}[ leftmargin=1cm ]
\item[\textbf{(CA)}] $(\cF,\mu)$ is \emph{contracting on average} on $(M,D)$ if there is $\lambda\in(0,1)$ such that for every $x,y\in M$, it holds
\[
	\int D(f_i(x),f_i(y))\,d\mu(i)
	\le  \lambda D(x,y).
\]
\end{itemize}
Note that, if $(M,D)$ is bounded, then \textbf{(CA)} implies the assertions in Theorem \ref{teo:sync-on-average}.

It is straightforward to check that for an RDS on a compact metric space and $\mu$ a probability measure on a compact metric space $\cF$, the property \textbf{(CA)} implies that the Markov operator $\bM$ defined in \eqref{cFast} is a contraction in the space of  Borel probability measures in the Wasserstein distance%
\footnote{Given two probability measures $\nu,\nu'$, their \emph{Wasserstein distance} or, sometimes also called, \emph{Hutchinson distance}  is defined by \[d_{\rm H}(\nu,\nu')\eqdef\sup\{\int f\,d\nu-\int f\,d\nu'\colon f\in{\rm Lip}_1(M)\}.\]
}
(see, for example, \cite[Theorem 1]{Men:98}). Hence, the uniqueness of the stationary measure is an immediate consequence of \textbf{(CA)}.

\begin{remark}
While \textbf{(CA)} is a sufficient condition to guarantee the uniqueness of the stationary measure, it is worth noting that it is not a necessary condition. An illustrative example is the system $(\cF,\mu)$ presented in Example \ref{exemplo:Anton}, where it can be readily demonstrated that a unique stationary measure exists (with $\mu$ being a non-degenerate measure on $\cI=\{1,2,3\}$). However, as the system does not satisfy proximality \textbf{(P)}, there is no metric having non-expansive or contractive behavior on average throughout the space.
\end{remark}

%------------------------------------------------------------------------------------------------------
\subsection{Limit laws}\label{subsec:INT-2}
%------------------------------------------------------------------------------------------------------
In this section, let us assume that $(\cF,\mu)$ has a unique stationary measure $\nu$.
Let us now discuss limit laws for an RDS which are consequences of the local contraction property \textbf{(LC)}. They describe the asymptotic behavior of the sum of results obtained from a large number of trials along a random orbit. That is, for a function $h\colon M\to\bR$, for every $x\in M$ and the associated Markov chain $(X_n^x)_{n\ge0}$ such laws describe the behavior of the sum
\[
S_n^x=\sum_{k=0}^{n-1}h\left(X_k^x\right)
\]
for $n$ large enough.

The strong law of large numbers states that $n^{-1}S_n^x$ tends to the expected value $\nu(h)$. The central limit theorem states that $n^{-1/2}S_n^x$ converges to a normal distribution
 with mean $\nu(h)$ and variance $\sigma^2(h)$. Recall that it allows the application of probabilistic methods that work for normal distributions, to characterize the asymptotic behavior of normalized sums along a random orbit. Finally, the law of the iterated logarithm claims that the standardized sum
\[
\frac{S_n^x-n\nu(h)}{\sqrt{n \sigma^2(h)}}
\]
approaches $\sqrt{2\log \log (n)}$ for $n$ large enough.

Let us formally establish those laws for additive functionals of Markov chains under fairly general hypotheses. Given a function $h\colon M\to\bR$,
\begin{enumerate}[ leftmargin=1.4cm ]
    \item[\textbf{(SLLN)}] $(\cF,\mu)$ satisfies the \emph{strong law of large numbers} for $h$ if for every $x\in M$,
    \[
    \lim_{n\to\infty}\frac{1}{n}\sum_{k=0}^{n-1} h\left(X_k^x\right)
    \overset{a.s.}{=}
    \nu(h)\eqdef\int h\,d\nu
    \quad\text{ as }n\to\infty.
    \]
    \item[\textbf{(CLT)}] $(\cF,\mu)$ satisfies the \emph{central limit theorem} for $h$ if for every $x\in M$,
    \[
    \frac{1}{\sqrt{n}}\sum_{k=0}^{n-1} h\left(X_k^x\right)
    \to \mathcal N(\nu(h),\sigma^2(h))
        \quad\text{ as }n\to\infty,
	\]
where
    \begin{equation}\label{defsigma2}
    \sigma^2(h)
    \eqdef \lim_{n\to\infty}\frac{1}{n}\int_{\cI^{\bN}\times M}\left(\sum_{k=0}^{n-1}
     h\left(X_k^x(\bi)\right)-n\nu(h)\right)^2 \,d\mu^\bN\otimes\nu(\bi,x).
    \end{equation}
    \item[\textbf{(LIL)}] $(\cF,\mu)$ satisfies the \emph{law of the iterated logarithm} for $h$ if for every $x\in M$,
    \[
    \limsup_{n\to\infty}\frac{1}{\sqrt{2n\log \log (n)}}\left(\sum_{k=0}^{n-1} h\left(X_k^x\right)-n\nu(h)\right)\overset{a.s.}{=}\sqrt{\sigma^2(h)}.
    \]
\end{enumerate}

Let us point out that the above limit laws hold for the Markov chain $(X_n^x)_{n\geq 0}$ at any initial point $x$. Sometimes such laws are also called \emph{quenched} laws. Note that, for example, \textbf{(SLLN)} is notably stronger than the Birkhoff ergodic theorem. Indeed, this theorem only implies convergence of $n^{-1}S_n^x(\bi)$ to the expected value $\nu(h)$ for $\mu^{\bN}\otimes\nu$-almost every $(\bi,x)$.

For any IFS that is contracting on average \textbf{(CA)}, the \textbf{(SLLN)} was established in \cite{Elt:87}. Here we extend those results to a general RDS, assuming compactness and local contraction. In another direction of generalizations, Lager\aa s and Stenflo in \cite{LagSte:05} consider contractive IFSs with place-dependent probabilities and establish a CLT.

\begin{theorem}\label{teo3.1-CLTbisneu}
	Assume that $(\cF,\mu)$ is an RDS of continuous maps on a compact metric space $(M,D)$ having the properties \textbf{(P)} and \textbf{(LC)}. Then $(\cF,\mu)$ satisfies the SLLN, the CLT, and the LIL  for every H\"older continuous function $h\colon M\to\bR$.
\end{theorem}

\subsection{Results for the case of IFSs on the circle}\label{sec:IFSoncircle}
In the remainder of this section, let us consider the particular case of an IFS with probabilities on $\bS^1$. Let $d$ be the usual metric on $\bS^1$, that is,
\begin{equation}\label{def:metusualcircle}
    d(x,y)\eqdef\min\{ |x-y|,1-|x-y|\}.
\end{equation}
Consider a finite collection of circle homeomorphisms $\cF=\{f_0,\ldots,f_{N-1}\}$. Let $\cI=\{0,\ldots,N-1\}$ and consider the Bernoulli probability measure $\mu^\bN$ induced by a probability vector $(p_0,\ldots,p_{N-1})$, that is, $\mu(\{i\})=p_i$. We will always assume that $\mu$ is non-degenerate, that is, that $p_0,\ldots, p_{N-1}>0$.
Note that the IFS $\cF$ on $\bS^1$ is never contracting on average with respect to the standard distance
(compare Figure \ref{fig.3} for relevant examples).

\begin{figure}[h]
 \begin{overpic}[scale=.5]{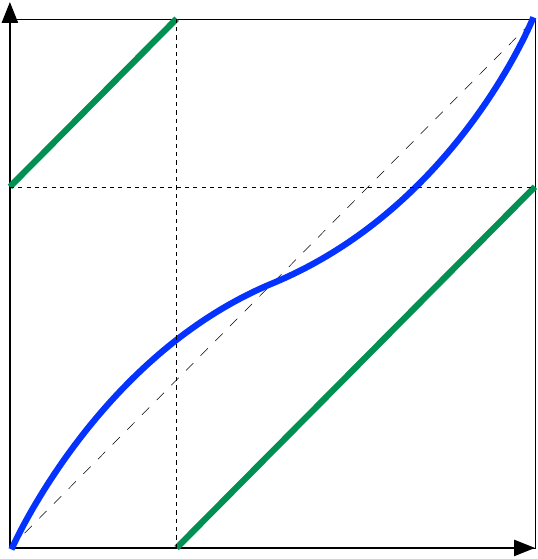}
	\put(97,67){$\textcolor{forestgreen}{f_0}$}
	\put(97,93){$\textcolor{blue}{f_1}$}
 \end{overpic}
 \hspace{2cm}
  \begin{overpic}[scale=.5]{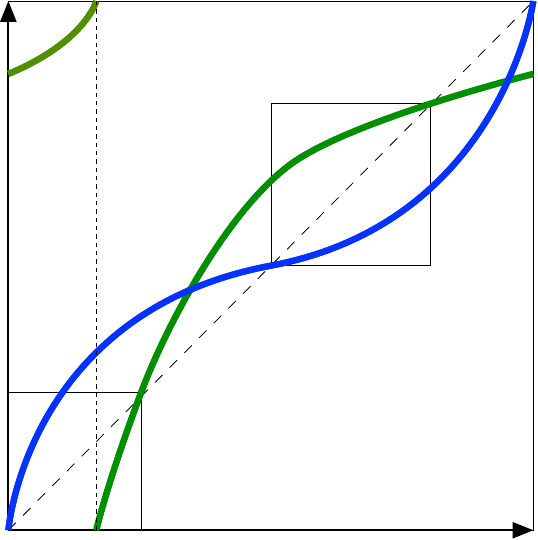}
	\put(101,83){$\textcolor{forestgreen}{f_0}$}
	\put(101,93){$\textcolor{blue}{f_1}$}
 \end{overpic}
  \caption{IFS $\cF$ which is forward minimal (left figure) and IFS $\cF$ which fails to be forward and backward minimal (right figure)}
 \label{fig.3}
\end{figure}

We will assume the following general hypothesis that is closely related to the local contraction property \textbf{(LC)} (see Remark \ref{reminterplay}):

\medskip

\begin{itemize}[ leftmargin=1cm ]
\item[\textbf{(H)}] there does not exist a probability measure which is \emph{invariant} by every map in $\cF$, that is, a measure $m$ which satisfies $f_\ast m=m$ for every $f\in\cF$.
\end{itemize}

\medskip

\begin{remark}\label{reminterplay}
The interplay of the concepts in \textbf{(H)} and \textbf{(P)} in a fairly general setting of random walks generated by a group of circle homeomorphisms is studied in \cite{Mal:17}. There Malicet adapts ideas of Crauel \cite{Cra:90} and Avila and Viana \cite{AviVia:10}, which in turn extend the so-called invariance principle (a contraction-invariance dichotomy) introduced already in \cite{Led:86}. By \cite[Theorem A]{Mal:17},
assuming that $\cF$ is an IFS of homeomorphism on $\bS^1$, hypothesis \textbf{(H)} implies \textbf{(LC)} for any non-degenerate probability vector. Moreover, assuming that $\cF$ satisfies \textbf{(LC)}, by \cite[Proposition 4.18]{Mal:17}, \textbf{(P)} and \textbf{(S)} are equivalent.
\end{remark}

The following is an immediate consequence of Theorem \ref{teo3.1-CLTbisneu}. It extends \cite{SzaZdu:21} which uses Malicet's assertion about \textbf{(LC)} in \cite{Mal:17}  to prove a CLT. In a similar spirit, \cite{CzuSza:20} prove a CLT for certain IFSs of homeomorphisms of the interval.

\begin{corollary}\label{cor-CLT}
Assume that $\cF=\{f_0,\ldots,f_{N-1}\}$ is an IFS of homeomorphisms on $(\bS^1,d)$ satisfying \textbf{(H)} and \textbf{(P)}. Then for any non-degenerated probability measure $\mu$ on $\cI=\{0,\ldots,N-1\}$,
the IFS with probabilities $(\cF,\mu)$ satisfies the SLLN, the CLT, and the LIL for every H\"older continuous function $h\colon\bS^1\to\bR$.
\end{corollary}

\begin{remark}[The role of minimality]\label{remminimality}
Another important key topological property, which is often assumed in this context, is minimality. Recall that the IFS $\cF$ acts \emph{forward minimally} if for every open interval $I\subset\bS^1$ and $x\in\bS^1$ there exist $\bi\in\cF^{\bN}$ and $n\in\bN$ such that $X_n^x(\bi)\in I$. Note that $\cF$ is forward minimal if and only if there is no nonempty closed set $A\subset\bS^1$, $A\ne\bS^1$, satisfying $f(A)\subset A$ for every $f\in \cF$.
Let us denote $\cF^{-1}\eqdef\{f_i^{-1}\}$, where $f_i\in\cF$. The IFS $\cF$ is \emph{backward minimal} if $\cF^{-1}$ is forward minimal. The existence of one (forward) minimal map in $\cF$ implies that the IFS acts forward minimally. Figure \ref{fig.3} (right figure) depicts an example for which $\cF$ is neither forward nor backward minimal.
Assuming minimality, a CLT and a LIL were obtained in \cite{SzaZdu:21,GabrielaTomasz:2022}.
On the other hand, assuming forward and backward minimality guarantees the syn\-chro\-nization-invariance dichotomy following Antonov \cite{Antonov:1984} (see \cite{GelSte:17} and references therein and also \cite{Mal:17} in a much wider setting). As one consequence, assuming forward and backward minimality, hypothesis \textbf{(H)} implies that $\cF$ has a topological factor IFS satisfying \textbf{(S)} (and hence \textbf{(P)}) (see \cite[Corollary 1]{GelSte:17}).
Here, we extend those results by proving limit laws in a context that is not \emph{a priori} minimal.
\end{remark}

Recall that for an IFS with probabilities satisfying \textbf{(CA)} (for any metric $D$), the stationary measure $\nu$ is unique (\cite[Theorem 2.1]{BDEG88}) and the distribution of $X_n^x$ converges exponentially fast to $\nu$ in the Prokhorov metric as $n\to\infty$. Moreover, this convergence is uniform on $x$ (\cite[Corollary 2.1]{JarnerTweedie}) and an Ergodic Theorem (see \cite[page 484]{Elt:87}), a SLLN and a CLT (see \cite[Theorem 5.1]{Pei:93}) hold.

Using metric change techniques, we derive the following result in the context of IFS with probabilities on the circle.

\begin{theorem}\label{teo3.1-CLT}
Assume that $\cF=\{f_0,\ldots,f_{N-1}\}$ is an IFS of $C^{1}$-diffeomorphisms on $(\bS^1,d)$ satisfying \textbf{(H)} and \textbf{(P)}. Then for every non-degenerate probability measure $\mu$ on $\cI=\{0,\ldots,N-1\}$ , $(\cF,\mu)$ has a unique stationary probability measure. Moreover, the SLLN and the CLT are satisfied for every H\"older function $h\colon\bS^1\to\bR$.
\end{theorem}

The assertion in Theorem \ref{teo3.1-CLT} about the uniqueness of the stationary measure is not new. To the best of our knowledge, a first reference under the same hypotheses is in \cite[Proposition 5.5.]{Deroin-Navas-Klepstyn:2007}. Furthermore, from results obtained in \cite{Mal:17}, we have that \textbf{(H)} and \textbf{(P)} imply synchronization and so there is a unique closed set which is simultaneously invariant by each map in $\cF$, therefore there is a unique stationary probability measure.
SLLN is then an immediate consequence of the Breiman theorem \cite{Bre:60}.
Theorem \ref{teo3.1-CLT} provides new sufficient conditions for the CLT to hold. It complements \cite[CLT Theorem 9]{SzaZdu:21} which is stated for IFS of circle homeomorphisms and replacing hypothesis \textbf{(P)} by the hypothesis that $(\cF,d)$ acts forward minimally.

%-----------------------------------------
\subsection{Large deviations of synchronization rates and Lyapunov exponents: IFSs on $\bS^1$}\label{subsec-INT-LD}
%-----------------------------------------

In this section, we return to the study of synchronization. We will consider the more specific case of an IFS with probabilities of continuous maps on $\bS^1$.
Applying Proposition \ref{prop:sync-exp}, we derive (in fact, exponentially fast) synchronization \textbf{(S)}. We will investigate its rates in more detail.

Synchronization has been studied in a wide range of contexts, including more general group actions of (orientation preserving) circle homeomorphisms \cite{Ghys:1999} or minimal IFSs (see \cite[Section 3.5.1]{Navas:2011}). By \cite[Theorem E]{Mal:17}, considering an IFS of circle homeomorphisms and assuming \textbf{(H)}, proximality \textbf{(P)} (and hence synchronization \textbf{(S)}) is equivalent to \emph{exponential synchronization}, that is, for every $x,y\in\bS^1$ and almost every $\bi\in\cI^{\bN}$ the sequence of numbers
\[
	 \big(d(X_n^x(\bi),X_n^y(\bi))\big)_{n\geq 0}
\]	
converges to $0$ exponentially fast as $n\to\infty$. In general, the rate of convergence depends on $x,y$, and $\bi$. Theorem \ref{teo:sync-on-average} asserts synchronization \emph{on average} with \emph{uniform} exponential decay rate. Below, we state large deviation results to estimate how fast those rates are approached.

Before stating these results, recall that the exponential rate of synchronization is intimately related to Lyapunov exponents, in particular in the context of diffeomorphisms. Let $\cF$ be an IFS of $C^1$-diffeomorphisms on $\bS^1$. Given a probability measure $\mu$ and $\nu$ it's associated (unique, by Theorem \ref{teo3.1-CLT}) stationary probability measure. By Birkhoff ergodic theorem, for $(\mu^\bN\otimes \nu)$-almost every $(\bi,x)\in\cI^\bN\times\bS^1$,
\[
	\lim_{n\to\infty}\frac{1}{n}\log|(f_{\bi}^n)'(x)|
	=\gamma(\mu)
	\eqdef \int\log|(f_{i_1})'(x)|\,d(\mu^{\bN}\otimes \nu)(\bi,x),
\]
where the number $\gamma(\mu)$ is also called \emph{fiber-Lyapunov exponent} of the IFS $(\cF,\mu)$. Note that if \textbf{(H)} holds, then $\gamma(\mu)<0$ (see, for example, \cite[Theorem F and Proposition 3.3]{Mal:17}).
We will also describe the deviation of finite-time Lyapunov exponents from $\gamma(\mu)$.

\begin{theorem}[Large deviations of Lyapunov exponents]\label{teolargedesvios}
Assume that $\cF=\{f_0,\ldots,f_{N-1}\}$ is an IFS of $C^{1+\beta}$ diffeomorphisms, for some $\beta>0$, on $\bS^1$ satisfying \textbf{(H)} and \textbf{(P)}. Then for every non-degenerate probability measure $\mu$ on $\cI=\{0,\ldots,N-1\}$ for every $x,y\in\bS^1$, $x\ne y$,
\[
    \lim_{n\to\infty}\frac{1}{n}\int\log d(X_n^x(\bi),X_n^y(\bi))\,d\mu^{\bN}(\bi) =\lim_{n\to\infty}\frac{1}{n}\int\log|(f_{\bi}^n)'(x)|\,d\mu^{\bN}(\bi)=\gamma(\mu)<0.
\]
Moreover, there exist positive constants $h,\varepsilon_0,c$ so that for all $x\in \mathbb{S}^1$, $\varepsilon\in(0,\varepsilon_0)$, and $n\in\mathbb{N}$
\[
	\mu^{\bN}\Big(\bi\in\cI^\bN\colon
	\big|\frac1n\log|(f_{\bi}^n)'(x)|-\gamma(\mu)\big|>\varepsilon\Big)
	\leq c e^{-nh\varepsilon^2},
\]
and for all $y\in\bS^1$, $y\neq x$,
\[
    \mu^{\bN}\left( \bi\in\cI^\bN\colon
    \Big|\frac1n\log \frac{d(X_n^x(\bi),X_n^y(\bi))}{d(x,y)}-\gamma(\mu)\Big|>\varepsilon\right)
    \le ce^{-n h\varepsilon^2/4}.
    \]
\end{theorem}

To prove Theorem \ref{teolargedesvios}, we study Markov systems and apply a result from \cite{DuaKle:17}, for which we need to consider H\"older potentials. For this reason, we additionally require $C^{1+\beta}$ regularity of the maps of the IFS.

%-----------------------------------------
\subsection{Structure of the paper}
%-----------------------------------------

In Sections \ref{sec:partIIneu}, we study the consequences of the property of local contraction \textbf{(LC)} and present the proofs of Proposition \ref{prop:sync-exp} and Theorem \ref{teo:sync-on-average}.
In Section \ref{sec33333}, we first invoke Theorem \ref{teo:sync-on-average} to conclude the \textbf{(CA)} property. This, in turn, implies the convergence of the arithmetic average of the transfer operator. Utilizing techniques from \cite{DerrLin:2003}, we prove the (CLT). Moreover, this allows us to verify a sufficient condition provided by \cite{DerrLin:2003} for establishing the (LIL). To prove the (SLLN), we investigate again finer consequences of the property \textbf{(LC)}. In Section \ref{sec:proooof}, we focus on the specific case of an IFS with probabilities on $\bS^1$. First, as a result of \cite{GelSal:}, the structure of the circle and the hypotheses \textbf{(H)} and \textbf{(P)} guarantee property \textbf{(CA)} with respect to the metric $d^\alpha$ for some $\alpha\in(0,1)$. Hence, one can invoke \cite{Pei:93} and conclude the limit laws SLLN and CLT asserted by Theorem \ref{teo3.1-CLT}. To prove the large deviation results in Theorem \ref{teolargedesvios}, we study Markov systems and invoke a result from \cite{DuaKle:17}.

%-----------------------------------------
\section{Systems satisfying \textbf{(LC)}: Linear cocycles}\label{sec:LC}
%-----------------------------------------

Let $\cI$ be a separable complete metric space and $\mu$ a Borel measure on $\cI$. Let $A\colon \cI\to\GL(d)$ be a measurable function such that $\log^+\lVert A^\pm\rVert\in L^1(\mu)$. Consider the one-sided
Bernoulli shift $\sigma$ on $(\cI^\bN,\mu^\bN)$. The \emph{linear cocycle} associated with $\mu$ is skew-product map $T:\cI^\bN\times\bR^d\to\cI^\bN\times\bR^d$ given by
\[
	T(\bi,v)\eqdef\big(\sigma(\bi),A(i_1)v\big).
\]
For $\bi=(i_1,i_2,\ldots)\in\cI^{\bN}$, denote
\[
A^n_{\bi}\eqdef A(i_n)\cdots A(i_1)\in \GL(d).
\]
Instead of the action of the matrices on $\bR^d$, we study their induced projective action on $\bP^{d-1}$. For $i\in\cI$, let $f_i$ be the projective map induced by the matrix $A(i)$.
\[
	f_i\colon\bP^{d-1}\to\bP^{d-1},
	\quad f_i(x)\eqdef\frac{A(i) x}{\Vert A(i) \Vert}.
\]
Let $D$ be the usual projective metric on $M=\bP^{d-1}$, that is, $D(x,y)=\Vert x\wedge y\Vert$ for $x,y\in M$. Consider the collection of continuous maps $\cF=\{f_{i}\}_{i\in\cI}$ on the compact metric space $(M,D)$. Consider the associated RDS $(\cF,\mu)$.

Denote by $\chi_1(\mu)\leq \chi_2(\mu)\leq\cdots\leq \chi_d(\mu)$ the fiber-Lyapunov exponents
of the RDS $(\cF,\mu)$. We refrain from giving the definition of them and just recall that, according to Kingman’s ergodic theorem, almost everywhere
\[
	\chi_k(\mu)
	=\lim_{n\to\infty}\frac{1}{n}\log s_k(A^n_{\bi}),
\]
where, for a $d\times d$-matrix $A$, the numbers $s_1(A) \geq s_2(A) \geq \cdots\geq s_d(A) \geq 0$ denote the singular values of $A$.
We say that $(\cF,\mu)$ is \emph{strongly irreducible} if there is no finite family of proper subspaces invariant by $A(i)$ for $\mu$-almost every $i\in\cI$. We say that $(\cF,\mu)$ has a \emph{simple top-Lyapunov exponent} if $\chi_d(\mu)>\chi_{d-1}(\mu)$.

\begin{lemma}\label{lem:ex-cocicles}
If  the RDS $(\cF,\mu)$
 is \emph{strongly irreducible} and has a simple top-Lyapunov exponent, then it has the local contraction property \textbf{(LC)} and the synchronization property \textbf{(S)} (and hence the property \textbf{(P)}).
\end{lemma}

It is well known in linear cocycle theory that the assumptions in Lemma \ref{lem:ex-cocicles} imply that the RDS $(\cF,\mu)$ has a
a unique $\mu$-stationary probability measure \textbf{(US)} (see \cite[Theorem 1]{GuiRau:1986}).

Before providing proof of this, note that there are many examples of linear cocycles satisfying these conditions, in particular, the linear cocycles studied in \cite[Section 8]{Viana:2014}.

\begin{proof}[Proof of Lemma \ref{lem:ex-cocicles}]
Since $\chi_d(\mu)>\chi_{d-1}(\mu)$,
by Oseledets multiplicative ergodic theorem, for $\mu^{\bN}$-almost every $\bi\in\cI^{\bN}$ there exists a $(d-1)$-dimensional subspace $V_{\bi}$ of $\bR^{d}$ such that for every $v\in\bR^{d}\setminus V_{\bi}$ and $w\in V_{\bi}$
\begin{equation}\label{eq:ex-cocy-furs-ex1-001}
\lim_{n\to\infty}\frac{1}{n}\log\Vert A^n_{\bi}v\Vert=\chi_d(\mu),\quad\mbox{and}\quad \lim_{n\to\infty}\frac{1}{n}\log\Vert A^n_{\bi}w\Vert<\chi_d(\mu).
\end{equation}
Denote by $\bP(V_{\bi})$ the projective subspace relative to $V_{\bi}$.

As $\cF$ is strongly irreducible, by \cite[Corollary 1.3, Sec. III]{Led:84}, for $\mu^{\bN}$-almost every $\bi\in\cI^{\bN}$,
\begin{equation}\label{eq:ex-cocy-furs-ex1}
\lim_{n\to\infty}\frac{1}{n}\log\Vert A^n_{\bi}x\Vert=\chi_d(\mu).
\end{equation}

Let us first prove the synchronization property \textbf{(S)}. Given $x,y\in\bP^{d-1}$ and $\bi\in\cI^{\bN}$,
\begin{align*}
    \limsup_{n\to\infty}\frac{1}{n}\log\delta(f_{\bi}^n(x),f_{\bi}^n(y))
    &=\limsup_{n\to\infty}\frac{1}{n}\log\frac{\Vert A^n_{\bi}x\wedge A^n_{\bi}y\Vert}{\Vert A^n_{\bi}x\Vert\,\Vert  A^n_{\bi}y\Vert}\\
    &\leq \lim_{n\to\infty}\frac{1}{n}\log\frac{s_d( A^n_{\bi})s_{d-1}( A^n_{\bi})}{\Vert A^n_{\bi}x\Vert\,\Vert  A^n_{\bi}y\Vert}.
\end{align*}
By \eqref{eq:ex-cocy-furs-ex1-001} and \eqref{eq:ex-cocy-furs-ex1}, for every $x,y\in\bP^{d-1}$ and for $\mu^{\bN}$-almost every $\bi\in\cI^{\bN}$,
\begin{align*}
    \limsup_{n\to\infty}\frac{1}{n}\log\delta(f_{\bi}^n(x),f_{\bi}^n(y))
    \leq\chi_d(\mu)+\chi_{d-1}(\mu)-2\chi_d(\mu)
    =\chi_{d-1}(\mu)-\chi_d(\mu)<0,
\end{align*}
which implies that \textbf{(S)} holds.

Now, let us prove the local contraction property \textbf{(LC)}. That is, let us prove that given $x\in\bP^{d-1}$ for $\mu^{\bN}$-almost every $\bi\in\cI^{\bN}$ there is a neighborhood $B_{\bi}\subset\bP^{d-1}$ of $x$ such that
\begin{equation}\label{eq1:ex-cocicles}
    \diam f_{\bi}^n(B_{\bi})\leq q^n,
\end{equation}
for $q=e^{(\chi_{d-1}-\chi_{d})/2}\in(0,1)$.
Fix $x\in\bP^{d-1}$. By \eqref{eq:ex-cocy-furs-ex1}, $x\notin \bP(V_{\bi})$. Since $\bP(V_{\bi})$ is a closed proper subset of $\bP^{d-1}$ and $x\notin \bP(V_{\bi})$, we have
\[
\beta(\bi)=\inf_{y\in\bP(V_{\bi})}D(x,y)>0.
\]
Let
\[
\hat{B}_\bi\eqdef\left\{y\in\bP^{d-1}\colon \delta(x,y)<\frac{\beta(\bi)}{2}\right\}.
\]
Now, let us show that
\begin{equation}\label{eq:ex-cocy-0009}
\lim_{n\to\infty}\frac{1}{n}\log \inf_{y\in \hat{B}_{\bi}}\left\Vert \frac{A^n_{\bi} y}{\Vert A^n_{\bi}\Vert} \right\Vert=0.
\end{equation}
Since $\Vert y\Vert=1$, the above limit is less than or equal to 0. To show the other inequality, possibly passing for a subsequence, we can assume that the sequence
\[
\left(\frac{1}{\Vert A^n_{\bi}\Vert} A^n_{\bi}\right)_{n\in\bN}
\]
converges to a matrix $M$ satisfying $\Vert M \Vert=1$. Note that $M v=0$ for every $w\in V_{\bi}$, because from \eqref{eq:ex-cocy-furs-ex1-001} we get
\[  \lim_{n\to\infty}\frac{1}{n}\log \left\Vert \frac{A^n_{\bi} w}{\Vert A^n_{\bi}\Vert} \right\Vert=0.\]
Hence, there exists $x^{\ast}\in\bP^{d-1}\setminus \bP(V_{\bi})$ such that $\Vert M x^{\ast}\Vert=1$.
By definition of $\hat{B}_{\bi}$, there exists $\alpha>0$ such that for every $y\in\hat{B}_{\bi}$ we have
$
\vert\langle x^{\ast},y \rangle\vert\geq \alpha.
$
By \eqref{eq:ex-cocy-furs-ex1-001},
\begin{align*}
\lim_{n\to\infty}\frac{1}{n}\log \inf_{y\in \hat{B}_{\bi}}\left\Vert \frac{A^n_{\bi} y}{\Vert A^n_{\bi}\Vert} \right\Vert&=\lim_{n\to\infty}\frac{1}{n}\left(\log \inf_{y\in \hat{B}_{\bi}}\vert\langle x^{\ast},y \rangle\vert+\log\left\Vert \frac{A^n_{\bi} x^{\ast}}{\Vert A^n_{\bi}\Vert} \right\Vert\right)\geq 0,
\end{align*}
which implies the desired inequality and so \eqref{eq:ex-cocy-0009}.

Since $\hat{B}_{\bi}\subset \bP^{d-1}\setminus \bP(V_{\bi})$, for every $y\in \hat{B}_\bi$, $y\neq x$, we have
\begin{align*}
    \limsup_{n\to\infty}\frac{1}{n}\log\diam f_{\bi}^n(\hat{B}_\bi)&\leq\limsup_{n\to\infty}\frac{1}{n}\log\sup_{y\in \hat{B}_\bi,\ y\neq x}\delta(f_{\bi}^n(x),f_{\bi}^n(y))\\
    &=\limsup_{n\to\infty}\frac{1}{n}\log\sup_{y\in \hat{B}_\bi,\ y\neq x}\frac{\Vert A^n_{\bi}x\wedge A^n_{\bi}y\Vert}{\Vert A^n_{\bi}x\Vert\,\Vert  A^n_{\bi}y\Vert}.
\end{align*}
By \eqref{eq:ex-cocy-furs-ex1} and \eqref{eq:ex-cocy-0009}, we get
\begin{align*}
    &\limsup_{n\to\infty}\frac{1}{n}\log\diam f_{\bi}^n(\hat{B}_\bi)\\
    &\leq \lim_{n\to\infty}\frac{1}{n}\log\frac{s_d\left( A^n_{\bi}\right)\, s_{d-1}\left( A^n_{\bi}\right)}{\Vert  A^n_{\bi}p\Vert\,\Vert A^n_{\bi}\Vert}-\liminf_{n\to\infty}\frac{1}{n}\log\inf_{y\in \hat{B}_\bi,\ y\neq x}\left\Vert \frac{A^n_{\bi}y}{\Vert A^n_{\bi}\Vert}\right\Vert\\
    &=\chi_d(\mu)+\chi_{d-1}(\mu)-2\chi_d(\mu)\\
    &=\chi_{d-1}(\mu)-\chi_d(\mu)<0,
\end{align*}
which implies \eqref{eq1:ex-cocicles} for some $B_{\bi}\subset \hat{B}_\bi$.
\end{proof}

%-----------------------------------------
\section{Exponential synchronization}\label{sec:partIIneu}
%-----------------------------------------

Throughout this section, $(M,D)$ is a compact metric space and $(\cF,\mu)$ is an RDS which satisfies \textbf{(LC)} and \textbf{(P)}. Let $q\in(0,1)$ be an associated contracting rate. Before proving Proposition \ref{prop:sync-exp} let us note the following consequence of the property \textbf{(LC)}:

\begin{lemma} \label{lemma:eq:diam-Omega}
	Assume that $(\cF,\mu)$ is an RDS on a compact metric space $(M,D)$ satisfying \textbf{(LC)}. Then for every $\delta\in(0,1)$ and $z\in M$ there is a neighborhood $B$ of $z$ so that
\begin{align*}
\mu^{\bN}\big(\big\{\bi\in\cI^{\bN}  \colon \mbox{diam}_D( f_\bi^n(B))
	\le q^n \mbox{ for all }n\in\bN\big\}\big)\geq \delta.
\end{align*}
\end{lemma}

\begin{proof}
Fix $z\in M$. For every $m\in\bN$ consider the open ball $B_m$ of radius $1/m$ and center $z$. Let $\Omega_m$ be the set of sequences which contract it, that is,
\[
   	\Omega_m
	\eqdef \big\{\bi\in\cI^{\bN} \colon \mbox{diam}_D\big( f_\bi^n(B_m)\big)
	\le q^n \mbox{ for all }n\in\bN\big\}
\]
Note that $\Omega_m\subset\Omega_{m+1}$ and hence,
\[
    \lim_{m\to\infty}\mu^{\bN}\big(\Omega_m \big)
    =\mu^{\bN}\big(\bigcup_{m\in\bN}\Omega_m \big).
\]
Now, let us show that
\begin{equation}\label{eq:claim-2.1-01}
       \mu^{\bN}\big(\bigcup_{m\in\bN}\Omega_m \big)=1.
\end{equation}
Indeed, by \textbf{(LC)}, there exists $\Gamma_z\subset \cI^\bN$ such that $\mu^\bN(\Gamma_z)=1$ and for every $\bi\in\Gamma_z$ there exists a neighborhood $N(\bi)$ of $z$ such that
\[
    \diam_D\big( f_\bi^n(N(\bi)\big)
	\le q^n \mbox{ for all }n\in\bN.
\]
Therefore, there is $m=m(\bi)\in\bN$ large enough so that $B_m\subset N(\bi)$ and hence $\bi\in\Omega_m$. This implies $\Gamma_z\subset \bigcup_{m\in\bN}\Omega_m$. From \eqref{eq:claim-2.1-01}, we get
    \[
    \lim_{m\to\infty}\mu^{\bN}\big(\Omega_m \big)=1.
    \]
    Hence, for any $\delta\in(0,1)$, we have that for $m$ large enough $\mu^{\bN}\big(\Omega_m \big)\geq \delta$. The assertion of the lemma follows by taking $B=B_m$.
\end{proof}

Now we are ready to prove Proposition \ref{prop:sync-exp}.

\begin{proof}[Proof of Proposition \ref{prop:sync-exp}]
The synchronization property \textbf{(S)} is a consequence of \cite[Proposition 4.18]{Mal:17}. Let us prove that the synchronization is exponential and that the rate of convergence is uniform. That is, we will prove that \eqref{prop1.1:syncexp} holds. Let us consider the following sets. Given $x,y\in M$, let
\begin{align*}
	\Omega^{x,y}
	&\eqdef \{\bi\in\cI^{\bN} \colon \text{ there is }n\in\bN\text{ such that }
				D(X_{n+k}^x(\bi),X_{n+k}^y(\bi))\le q^k\text{ for all }k\geq 0\},\\
      \mathcal{S}^{x,y}&\eqdef\{\bi\in\cI^{\bN}  \colon \lim_{n\to\infty}D(X_n^x(\bi),X_n^y(\bi))=0\}.
\end{align*}
By \textbf{(S)}, $\mu^\bN(\mathcal{S}^{x,y})=1$ for every $x,y\in M$. It remains to prove the following claim which immediately will imply \eqref{prop1.1:syncexp}.

\begin{claim}\label{claim:omegafullmeas}
	For every $x,y\in M$, $\mu^\bN(\Omega^{x,y})=1$.
\end{claim}

\begin{proof}
Fix $x,y\in M$ and $\delta\in(0,1)$. For simplicity of notation, for the remainder of this proof, we continue to write $\Omega=\Omega^{x,y}$ and $\mathcal{S}=\mathcal{S}^{x,y}$. By Lemma \ref{lemma:eq:diam-Omega}, for every $z\in M$, for almost every $\bi$ there exists a neighborhood $B_z$ of $z$ such that
\begin{align}\label{eq:diam-Omega}
\mu^{\bN}\big(\big\{\bi\in\cI^{\bN}  \colon \mbox{diam}_D( f_\bi^n(B_z))
	\le q^n \mbox{ for all }n\in\bN\big\}\big)\geq \delta.
\end{align}
As $(M,D)$ is compact, there exist $z_1,\ldots,z_m\in M$ such that, letting $B_j=B_{z_j}$, we have $M\subset\bigcup_{j=1}^m B_j$. These sets can intersect. Let us denote by
\[
	\mathcal{B}\eqdef\Big\{\bigcap_{j\in J} B_j\colon J\subset \{1,\ldots,m\} \Big\}
\]
the family of all such intersections. Note that this is some finite collection of open sets $\mathcal{B}=\{V_1,\ldots,V_k\}$. For every $\ell\in\{1,\ldots,k\}$ there is $j\in\{1,\ldots,m\}$ so that $V_\ell\subset B_j$.  By \eqref{eq:diam-Omega}, for every $\ell\in\{1,\ldots,k\}$ there is a set $\mathcal{C}_\ell\subset \cI^\bN$ such that
\begin{equation}\label{eq:measure-C_i}
   \mu^{\bN}(\mathcal{C}_\ell)\geq \delta,
   \quad\text{where}\quad
   \mathcal{C}_\ell\eqdef\left\{\bi\in\cI^{\bN}  \colon \mbox{diam}_D\left( f_\bi^n(V_\ell)\right)
	\le q^n \mbox{ for all }n\in\bN\right\}.
\end{equation}
Now, note that for any $\bi \in \mathcal{S}$ there exist $n\in\bN$ and $j\in\{1,\ldots,m\}$ such that
$f_{\bi}^n(x)$ and $f_{\bi}^n(y)$ both are in $B_j$ and denote by $N(\bi)$ the smallest $n$ with this property, that is
\[
N(\bi)\eqdef\min\{n\in\bN\colon f_{\bi}^n(x),f_{\bi}^n(y)\in B_j, \mbox{ for some }j=1,\ldots,m\}.
\]
Note that \emph{a priori} there can be more than one index $j$ with this property. Therefore, let
\[
B(\bi)\eqdef\bigcap_{j\colon f_{\bi}^{N(\bi)}(x),f_{\bi}^{N(\bi)}(y)\in B_j} B_j \in\mathcal{B}.
\]
 It is not difficult to show that $B\colon \mathcal{S} \to \mathcal{B}$ and $N\colon \mathcal{S}\to \bN$ are measurable functions. Setting
\[
	\mathcal{S}(n,\ell)
	\eqdef\{\bi\colon N(\bi)=n,\,B(\bi)=V_\ell\},
\]
we have that
\begin{equation}\label{unionunion}
\mathcal{S}=\bigcup_{n\in\bN}\bigcup_{\ell=1,\ldots,k}\mathcal{S}(n,\ell)
\end{equation}
is a disjoint union of measurable sets.

Note that if $\bi=(i_1,i_2,\ldots)\in\mathcal S(n,\ell)$, then every $\bj=(j_1,j_2,\ldots)$ satisfying $ j_k=i_k$ for all $k=1,\ldots,n$ is also in this set. Let
\[
	\mathcal S(n,\ell)_1^n
	\eqdef\{(\bi_1,\ldots,\bi_n)\colon (\bi_1,\ldots,\bi_n)\times \cI^{\bN}\subset \mathcal{S}(n,\ell)\}
	\subset \cI^n.
\]
It is not hard to see that $\mathcal S(n,\ell)_1^n\times\mathcal{C}_\ell\subset \Omega$. In particular,
\[
\bigcup_{n\in\bN}\bigcup_{\ell=1,\ldots,k}\mathcal S(n,\ell)_1^n\times\mathcal{C}_\ell\subset \Omega.
\]
Note that this is also a disjoint union of measurable sets.
By \eqref{eq:measure-C_i}, we have
\[
	\mu^\bN\big(\mathcal S(n,\ell)_1^n\times\mathcal{C}_\ell\big)
	=\mu^n(\mathcal S(n,\ell)_1^n)\cdot\mu^\bN(\mathcal{C}_j)
	\geq \mu^n(\mathcal S(n,\ell)_1^n)\cdot\delta.
\]
We get
\begin{align*}
    \mu^{\bN}(\Omega)
    &\geq \sum_{n\in\bN} \sum_{\ell=1,\ldots,k}\mu^\bN\big(\mathcal S(n,\ell)_1^n\times\mathcal{C}_\ell\big)\\
    &\geq \delta  \sum_{n\in\bN} \sum_{\ell=1,\ldots,k}\mu^n\big(\mathcal S(n,\ell)_1^n\big)
    =\delta  \sum_{n\in\bN} \sum_{\ell=1,\ldots,k}\mu^{\bN}\big(\mathcal S(n,\ell)_1^n\times \cI^{\bN}\big)\\
    &\geq\delta  \sum_{n\in\bN} \sum_{\ell=1,\ldots,k}\mu^{\bN}\big(\mathcal{S}(n,\ell)\big)
    =\delta \mu^{\bN}(\mathcal{S})
    =\delta,
\end{align*}
where we also used \eqref{unionunion} and $\mu^{\bN}(\mathcal{S})=1$.
This implies that $\mu^{\bN}(\Omega)\geq \delta$. As $\delta\in(0,1)$ is arbitrary, it follows that $\Omega^{x,y}=\Omega$ has full measure.
\end{proof}

To complete the proof of the proposition, we will establish the uniqueness of the stationary measure. By synchronization and dominated convergence theorem, for every $x,y\in M$
\[
\lim_{n\to\infty}\int_{\cI^{\bN}}D(X_n^x(\bi),X_n^y(\bi))\, d\mu^\bN(\bi)=0.
\]
Let $\nu_1$ and $\nu_2$ be two $\mu$-stationary measures. Let $h\colon(M, D)\to \bR$ be a Lipschitz function. Then, for all $n\in \bN$
\begin{align*}
    &\left\vert \int_M h(x)\,d\nu_1(x)- \int_M h(y)\,d\nu_2(y) \right\vert\\
    &=\left\vert \int_M \int_{\cI^{\bN}}h(X_n^x(\bi))\,d\mu^\bN(\bi)\,d\nu_1(x)- \int_M\int_{\cI^{\bN}}h(X_n^x(\bi))\,d\mu^\bN(\bi)\,d\nu_2(y) \right\vert\\
    &\leq \int_M \int_M\left(\int_{\cI^{\bN}}\left\vert h(X_n^x(\bi)) - h(X_n^y(\bi))\right\vert \,d\mu^\bN(\bi)\right)d\nu_1(x)\,d\nu_2(y)\\
    &\leq \mbox{Lip}(h)\int_M \int_M\left(\int_{\cI^{\bN}} D(X_n^x(\bi),X_n^y(\bi)) 	\,d\mu^\bN(\bi)\right)d\nu_1(x)\,d\nu_2(y).
\end{align*}
Letting $n$ tend to infinity and again applying the dominated convergence theorem, we find that the right-hand side converges to 0.
Therefore,
\begin{equation}\label{equalitu-stationary-RD}
\int_M h\,d\nu_1= \int_M h\,d\nu_2.
\end{equation}
By compactness of $(M, D)$, the set of the Lipschitz functions is dense in the set of the continuous functions. This, together with the dominated convergence theorem, implies that \eqref{equalitu-stationary-RD} holds for any continuous function $h\colon M\to \bR$. Hence,   $\nu_1=\nu_2$ and \textbf{(US)} follows.

This finishes the proof of the proposition.
\end{proof}

\begin{proof}[Proof of Corollary \ref{cor:atomic}]
Let $\nu$ be the $\mu$-stationary probability measure on $M$. Assume that $\nu$ has atoms and let
\[
    \beta
    \eqdef\max_{x\in M}\nu(\{x\})
    \quad\text{ and }\quad
    E\eqdef\left\{x\in M\colon \nu(\{x\})=\beta\right\}.
\]
Since $\nu $ is a probability, $E$ has finite cardinality. As $\nu$ is stationary, for any $x\in E$
\begin{align}\label{eq:cor-atom}
    \beta=\nu(\{x\})=\int \nu\left((f_i)^{-1}(\{x\})\right)\,d\mu(i).
\end{align}
Since all maps in $\cF$ are injective, $(f_i)^{-1}(\{x\})$ has at most one element. For $x\in E$, it holds $\nu\left((f_i)^{-1}(\{x\})\right)\leq \beta$. Moreover, \eqref{eq:cor-atom} implies that for $\mu$-almost every $i$ we have $\nu\left((f_i)^{-1}(\{x\})\right)= \beta$. In particular, $(f_i)^{-1}(E)\subset E$ and hence $\card (f_i)^{-1}(E)\leq \card  E$. Again by stationarity of $\nu$,
\[
    \card E\cdot \beta
    =\nu(E)=\int \nu\left((f_i)^{-1}(E)\right)\,d\mu(i)=\beta\int \card (f_i)^{-1}(E)\,d\mu(i),
\]
    we get that $\card (f_i)^{-1}(E)= \card  E$ for $\mu$-almost every $i\in\cI$. Therefore,
    \[
    E=(f_i)^{-1}(E)=f_i(E)
    \]
    for $\mu$-almost every $i\in\cI$.
Let us show that $\card E=1$. Indeed, otherwise,
\[
    r
    \eqdef\min\{D(x,y)\colon x\neq y,\ x,y\in E\}>0
\]
and for any $x,y\in E$, $x\neq y$, for every $n\in\bN$ and $\mu^{\bN}$-almost every $\bi\in\cI^{\bN}$
\[
    D(f_{\bi}^n(x),f_{\bi}^n(y))\geq r>0,
\]
which contradicts the synchronization \textbf{(S)} established in Proposition \ref{prop:sync-exp}.
This proves $E=\{x_0\}$ for some $x_0\in M$. By \textbf{(S)}, for every $y\in M$ and $\mu^{\bN}$-almost every $\bi\in\cI^{\bN}$
\[
    \lim_{n\to\infty}D(f_{\bi}^n(y),x_0)=0.
\]
In other words, random orbits accumulate at $x_0$ and hence the stationary measure $\nu$ is the Dirac measure at $x_0$. This proves the assertion.
\end{proof}

For the proof of the following corollary, recall that an RDS $(\cF,\mu)$ is \emph{aperiodic} if there does not exist a finite number $p\ge2$ of pairwise disjoint closed subsets $F_1,\ldots, F_p\subset\bS^1$ so that for $\mu$-almost every $i$, $f_i(F_k)\subset F_{k+1}$ for $k=1,\ldots, p-1$ and $f_i(F_p)\subset F_1$.

\begin{corollary}
	Assume that $(\cF,\mu)$ is an RDS of continuous maps on a compact metric space $(M,D)$ satisfying \textbf{(P)} and \textbf{(LC)}. Then for any continuous function $\varphi\colon M\to\bR$,
\[
	\lim_{n\to\infty}\sup_{x\in M}\left|\int\varphi(X_n^x(\bi))\,d\mu^{\bN}(\bi)
	- \int\varphi\,d\nu\right|=0.
\]
In other words, for every $x\in M$ the distribution of $X_n^x$,
\begin{equation}\label{eqPort}
\left(\mu^{\bN}(\{\bi\in\cI^{\bN}\colon X_n^x(\bi)\in \cdot\})\right)_n,
\end{equation}
converges (uniformly) in the weak topology to $\nu$.	
\end{corollary}

\begin{proof}
	Note that aperiodicity is an immediate consequence of property \textbf{(P)}.
Hence, we can invoke \cite[Proposition 4.14]{Mal:17} that implies the assertion.
\end{proof}

We are going to need one more ingredient to prove the Theorem \ref{teo:sync-on-average}. 
By the Portmanteau theorem applied to the sequence of distributions \eqref{eqPort} and any measurable set $B\subset M$ satisfying $\nu(B)>0$ and $\nu(\partial B)=0$, for every $x\in M$ we get
\begin{equation}\label{converweak:point}
	 \lim_{n\to\infty}\big|\mu^{\bN}(\{\bi\in\cI^{\bN}\colon X_n^x(\bi)\in B\})-\nu(B)\big|=0.
\end{equation}
Let us now derive some slightly stronger convergence than the one in \eqref{converweak:point}.

\begin{proposition}\label{proclaim:conunifweak}
For every measurable set  $B\subset M$ such that $\nu(B)>0$ and $\nu(\partial B)=0$,
\[	\lim_{n\to\infty}\sup_{x,y\in M}
		\big|\mu^\bN(\{\bi\colon X_n^x(\bi),X_n^y(\bi)\in B\})-\nu(B)\big|=0.\]
\end{proposition}

\begin{proof}
Let $B\subset M$ be a set satisfying the hypotheses.
By contradiction, suppose that there exist $\delta\in(0,1)$ and some subsequence $(n_k)_k$  such that for all $k\in\bN$
\[
	\sup_{x,y\in M}\big|\mu^{\bN}(\{\bi\colon X_{n_k}^x(\bi),X_{n_k}^y(\bi) \in B\})-\nu(B)\big|
	> 2\delta.
\]
Then there exist two sequences $(x_k)_k$ and $(y_k)_k$ in $M$ such that for all $k\in\bN$
\begin{equation}\label{contrad}
	\big|\mu^{\bN}(\{\bi\colon X_{n_k}^{x_k}(\bi),X_{n_k}^{y_k}(\bi)\in B\})-\nu(B)\big|
	> \delta.
\end{equation}	
By compactness, up to possibly passing to some subsequence, we can assume that $x_k\to x$ and $y_k\to y$ as $k\to\infty$ for some $x,y\in M$.
By the triangle inequality, for all $k\in\bN$
\[\begin{split}
	\big|\mu^{\bN}&(\{\bi\colon X_{n_k}^{x_k}(\bi),X_{n_k}^{y_k}(\bi)\in B\})-\nu(B)\big|\\
&\leq \big|\mu^{\bN}(\{\bi\colon X_{n_k}^{x_k}(\bi),X_{n_k}^{y_k}(\bi)\in B\})-\mu^{\bN}(\{\bi\colon X_{n_k}^{x}(\bi)\in B\})\big|\\
&\phantom{\leq}+\big\lvert\mu^{\bN}(\{\bi\colon X_{n_k}^{x}(\bi))\in B\})-\nu(B)\big|.
\end{split}\]
To contradict \eqref{contrad} let us prove that the right-hand side in the above formula converges to $0$ as $k\to\infty$. Indeed, by \eqref{converweak:point}, the latter term converges to 0. For the first term, we need to investigate those realizations that simultaneously are either in $B$ or are not in $B$. For that, for    $\varepsilon>0$ sufficiently small consider the set $B_\varepsilon\eqdef\{x\in M\colon D(x,\partial B)>\varepsilon\}$ of points that are $\varepsilon$-distant from its boundary. 
Below, we will use our hypothesis $\nu(\partial B)=0$ in order to estimate the following terms separately
\begin{eqnarray}
&&\big|\mu^{\bN}(\{\bi\colon X_{n_k}^{x_k}(\bi)\in B\})-\nu(B)\big|\le\notag\\
&&\quad\quad\quad \big|\mu^{\bN}(\{\bi\colon X_{n_k}^{x}(\bi)\in B\})-\nu(B)\big|  \label{destriang:diracs-0}\\
&&\quad\quad\quad\quad
	+\int_{\{\bi\colon X_{n_k}^{x}(\bi)\in B_\varepsilon\}}
	\big|\mathbbm{1}_{B}( X_{n_k}^{x_k}(\bi))\mathbbm{1}_{B}( X_{n_k}^{y_k}(\bi))-\mathbbm{1}_{B} (X_{n_k}^{x}(\bi))\big|\,d\mu^{\bN}(\bi) \label{destriang:diracs-1}\\
&&\quad\quad\quad\quad
	+\int_{\{\bi\colon X_{n_k}^{x}(\bi)\not\in B_\varepsilon\}}\big|\mathbbm{1}_{B}( X_{n_k}^{x_k}(\bi))\mathbbm{1}_{B}( X_{n_k}^{y_k}(\bi))-\mathbbm{1}_{B} (X_{n_k}^{x}(\bi))\big|\,d\mu^{\bN}(\bi).\label{destriang:diracs-2}
\end{eqnarray}

Note again that, by \eqref{converweak:point}, the term in \eqref{destriang:diracs-0} tends to $0$ as $k\to\infty$.

Let $q\in(0,1)$ be an associated contracting rate with respect to \textbf{(LC)}.
By \textbf{(LC)} applied to $x$ and $y$, respectivamente, for almost every $\bi$ there exists $k_0(\bi)\in\bN$ such that for all $k\ge k_0(\bi)$ ,
\[
	\max\{D\big(X_{n_k}^{x_k}(\bi),X_{n_k}^{x}(\bi)\big),D\big(X_{n_k}^{y_k}(\bi),X_{n_k}^{y}(\bi)\big)
	\}\leq q^{n_k}.
\]
Moreover, by Proposition \ref{prop:sync-exp} applied to $x$ and $y$, for almost every $\bi$ there exist $C(\bi)>1$ such that for all $k>0$
\[
	D\big(X_{n_k}^{x}(\bi),X_{n_k}^{y}(\bi)\big)
	\leq C(\bi) q^{n_k}.
\]
Hence, for almost every $\bi$ for $k\ge k_0(\bi)$ large enough such that $C(\bi)q^{n_k}<\varepsilon/2$, we have
\[
	\max\{D\big(X_{n_k}^{x_k}(\bi),X_{n_k}^{x}(\bi)\big),
		D\big(X_{n_k}^{y_k}(\bi),X_{n_k}^{y}(\bi)\big),
		D\big(X_{n_k}^{x}(\bi),X_{n_k}^{y}(\bi)\big)\}
	\leq C(\bi)q^{n_k}
	<\frac\varepsilon2.
\]
Note that whenever $X_{n_k}^{x}(\bi)\in B_\varepsilon$ by the above, for almost every $\bi$ and large enough $k$ either both $X_{n_k}^{x_k}(\bi),X_{n_k}^{y_k}(\bi)\in B$ or $\not\in B$ and hence we have 
\[
\big|\mathbbm{1}_{B}( X_{n_k}^{x_k}(\bi))\mathbbm{1}_{B}( X_{n_k}^{y_k}(\bi))
	-\mathbbm{1}_{B} (X_{n_k}^{x}(\bi))\big|=0;
\]
Hence, it follows from the Dominated convergence theorem, that
\[
	\lim_{k\to\infty}\int_{\{\bi\colon X_{n_k}^{x}(\bi)\in B_\varepsilon\}}
	\big|\mathbbm{1}_{B}( X_{n_k}^{x_k}(\bi))\mathbbm{1}_{B}( X_{n_k}^{y_k}(\bi))
		-\mathbbm{1}_{B} (X_{n_k}^{x}(\bi))\big|\,d\mu^{\bN}(\bi)
	=0.
\]
This proves the limit behavior of \eqref{destriang:diracs-1} as $k\to\infty$.

To estimate \eqref{destriang:diracs-2}, note that by the Portmanteau theorem applied to the distribution \eqref{eqPort} and the closed set $M\setminus B_\varepsilon$, we get
\[\begin{split}
    \limsup_{k\to\infty}\int_{\{\bi\colon X_{n_k}^{x}(\bi)\not\in B_\varepsilon\}}
    	&\big|\mathbbm{1}_{B}( X_{n_k}^{x_k}(\bi))\mathbbm{1}_{B}( X_{n_k}^{y_k})
		-\mathbbm{1}_{B} (X_{n_k}^{x}(\bi))\big|\,d\mu^{\bN}(\bi)\\
	&\leq \limsup_{k\to\infty}\mu^{\bN}\left(\{\bi\colon X_{n_k}^{x}(\bi)\notin B_\varepsilon\}\right)\\
	&\leq \nu\left( M\setminus B_\varepsilon\right).
\end{split}\]
Letting $\varepsilon\to0$ and using $\nu(\partial B)=0$, we get $\nu\left( M\setminus B_\varepsilon\right)\to0$. As $\varepsilon>0$ was arbitrary in our above arguments, this together implies that
\[
\lim_{k\to\infty}\big|\mu^{\bN}(\{\bi\colon X_{n_k}^{x_{n_k}}(\bi)\in B\})-\nu(B)\big|=0,
\]
proving the limit behavior of \eqref{destriang:diracs-2}. 

This together is in contradiction to \eqref{contrad}. This finishes the proof of the proposition.
\end{proof}

We are now ready to prove Theorem \ref{teo:sync-on-average}. 

\begin{proof}[Proof of Theorem \ref{teo:sync-on-average}]
Assume that $(\cF,\mu)$ satisfies \textbf{(P)} and \textbf{(LC)} and as above let $q\in(0,1)$ be an associated contraction rate. Fix some point $x^{\ast}\in\supp\nu$. Choose $r\in(0,1)$ sufficiently small such that
\begin{equation}\label{eq:r-ast}
    \nu\left(\left\{x\in M\colon D(x,x^{\ast})=r\right\}\right)=0
\end{equation}
and hence, by this choice, the set $B\eqdef\{x\colon D(x,x^\ast)<r\}$ satisfies $\nu(B)>0$ and $\nu(\partial B)=0$.
By Lemma \ref{lemma:eq:diam-Omega}, we can assume that $r $ is small enough to have
\[	\delta\eqdef\mu^\bN(\Omega)>0
	\quad\text{ where }\quad
	\Omega 
	\eqdef\big\{\bi\in\cI^{\bN}  \colon \mbox{diam}_D( f_\bi^n(B))
	\le q^n \mbox{ for all }n\in\bN\big\}.
\]

Fix $\alpha\in(0,1)$. For every $n\in\bN$, let
\begin{equation}\label{def:an}
	a_n \eqdef \sup_{x,y \in M}  \sum_{m=1}^{n} \int D^\alpha(X_m^x(\bi),X_m^y(\bi))\,d\mu^\bN(\bi).
\end{equation}
Note that $(a_n)_n$ is increasing. 

For what is below, denote by $\pi_n(\Omega)$ the projection of $\Omega$ on $\cI^n$, that is,
\[
	\pi_n(\Omega)
	\eqdef \big\{ (i_1, \ldots, i_n) \in \cI^n \colon
	 (i_1, \ldots, i_n) \times \cI^\bN \cap \Omega \neq \emptyset \big\}.
\]

Let us show the following claim.

\begin{claim}\label{someclaim}
	For every $z, w \in B$ and $n\in\bN$,
\[
	\sum_{m=1}^{n} \int D^\alpha(X_m^w(\bi),X_m^z(\bi))\,d\mu^\bN(\bi)
	\le 2(q^{\alpha}+q^{2\alpha}+\cdots+q^{n\alpha})+a_n(1-\delta).
\]	 
\end{claim}

\begin{proof}
To prove the claim, let us first write
\begin{multline*}
    \sum_{m=1}^{n}
     \int D^\alpha(X_m^w(\bi),X_m^z(\bi))\,d\mu^\bN(\bi)\\
    =\int_{\pi_n(\Omega)}\sum_{m=1}^{n} D^\alpha(X_m^w(\bi),X_m^z(\bi))\,d\mu^n(\bi)
    +	\int_{\cI^n\setminus \pi_n(\Omega)}\sum_{m=1}^{n} D^\alpha(X_m^w(\bi),X_m^z(\bi))\,d\mu^n(\bi)\\
	 =\text{I}+\text{II}.
\end{multline*}

Let us first estimate integral II. For \( k \leq n \), consider the set
\[
G_n^k \eqdef \left\{ (i_1, \ldots, i_n) \in \cI^n \colon (i_1, \ldots, i_{k-1}) \in \pi_{k-1}(\Omega), (i_1, \ldots, i_k) \notin \pi_{k}(\Omega) \right\}
\]
of all finite sequences which in their first \( k - 1 \) elements coincide with some sequence
in \( \Omega \), but in their first \( k \) elements do not. Note that
\[
	\cI^n\setminus \pi_n(\Omega) = \bigcup_{k=1}^{n} G_n^k 
	\quad\text{ and }\quad 
	G_n^k \cap G_n^\ell = \emptyset \text{ for $k\neq \ell$}.
\]
We easily see that
\begin{equation}\label{eq:med:ast}
	\mu^n(\cI^n\setminus \pi_n(\Omega)) 
	\leq \mu^{\bN}(\cI^\bN\setminus \Omega) 
	= 1 - \mu^{\bN}(\Omega) 
	= 1 - \delta.
\end{equation}
Further, let \( G_n^{k \vert k} \) be the projection of \( G_n^k \) on its first \( k \) coordinates, i.e.
\[
G_n^{k \vert k} \eqdef \left\{ (i_1, \ldots, i_k) \in \cI^k : \exists (i_1, \ldots, i_k, i_{k+1}, \ldots, i_n) \in G_n^k \right\}.
\]
Note that the sets \( G_n^k \), \( k = 1, \ldots, n \), are pairwise disjoint. 
For the following estimate, note that the mapping $\bi\mapsto X_m^w(\bi)$ in fact only depends on the first $m$ coordinates of $\bi=(i_1,\ldots,i_m,i_{m+1},\ldots)$. Hence, to simplify our exposition, by a slight abuse of notation, in the following let us write $X_m^w(\bi)=X_m^w(i_1,\ldots,i_n)$, for any $n\ge m$.
Moreover, \( G_n^k = G_n^{k \vert k} \times \cI^{n-k} \) and hence, by Fubini’s theorem,
\begin{align*}
 \text{II}&=\sum_{k=1}^n\int\limits_{G_n^k}\, \sum_{m=1}^{n} 
 	D^\alpha(X_m^w(\bi),X_m^z(\bi))\,d\mu^n(\bi)\\
 &=\sum_{k=1}^n\int\limits_{G_n^{k \vert k}} \int\limits_{\cI^{n-k}}\, 
 	\sum_{m=1}^{n} 
 	D^\alpha(X_m^w(i_1, \ldots, i_n),X_m^z(i_1, \ldots, i_n))\,\\
 &\quad\quad\quad\quad\quad\quad\quad d\mu^{n-k}( i_{k+1}, \ldots, i_n)d\mu^{k}(i_1, \ldots, i_k).
\end{align*}
For \( k = 1, \ldots, n \), we have
\begin{align*}
 &\int\limits_{G_n^{k \vert k}} \int\limits_{\cI^{n-k}}\hspace{-0.1cm} \sum_{m=1}^{n} 
 D^\alpha\big(X_m^w(i_1, \ldots, i_n),X_m^z(i_1, \ldots, i_n)\big)\,
 	d\mu^{n-k}( i_{k+1}, \ldots, i_n)d\mu^{k}(i_1, \ldots, i_k)\\
&=\int\limits_{G_n^{k \vert k}} \int\limits_{\cI^{n-k}}\left( \sum_{m=1}^{k} 
 D^\alpha\big(X_m^w(i_1, \ldots, i_n),X_m^z(i_1, \ldots, i_n)\big)\right.\\
&+\left.
	 \sum_{m=k+1}^{n} 
 D^\alpha(X_m^w(i_1, \ldots, i_n),X_m^z(i_1, \ldots, i_n))\right)
 	d\mu^{n-k}( i_{k+1}, \ldots, i_n)d\mu^{k}(i_1, \ldots, i_k)		\\
 &=\int\limits_{G_n^{k \vert k}} \, \sum_{m=1}^{k} 
 D^\alpha\big(X_m^w(i_1, \ldots, i_k),X_m^z(i_1, \ldots, i_k)\big)\,
 	d\mu^{k}(i_1, \ldots, i_k)\\
 & \quad \quad\quad \quad  
 +\int\limits_{G_n^{k \vert k}} \int\limits_{\cI^{\bN}}\, \sum_{\ell=1}^{n-k} 
 	D^\alpha\big(X_\ell^{f_{i_k}\circ\cdots\circ f_{i_1}(w)}(\bj),X_\ell^{f_{i_k}\circ\cdots\circ f_{i_1}(z)}(\bj)\big)\,
	d\mu^{\bN}( \bj)d\mu^{k}(i_1, \ldots, i_k)
\end{align*}
Recall now that $z,w\in B$ and that $\Omega$ is the set of sequences $\bi$ which contract $B$ (see \eqref{recall}).  Hence, we get
\begin{align*}	
 \text{II}
 &\leq\sum_{k=1}^n (q^\alpha+q^{2\alpha}+\cdots+q^{k\alpha})\mu^k(G_n^{k \vert k})
 	+  \sum_{k=1}^n a_{n-k}\int_{G_n^{k \vert k}} d\mu^{k}(i_1, \ldots, i_k)\\
\text{\tiny{(using $a_{n-k}<a_n$)}}	
&\le 	(q^\alpha+q^{2\alpha}+\cdots+q^{n\alpha})\sum_{k=1}^n\mu^k(G_n^{k \vert k})
	+a_n\sum_{k=1}^n \mu^k(G_n^{k \vert k})\\
 &\leq (q^\alpha+q^{2\alpha}+\cdots+q^{n\alpha}+a_n)\mu^n(G_n^{k}).
\end{align*}
By \eqref{eq:med:ast}, we conclude
\[
\text{II}\leq(1-\delta)(q^{\alpha}+q^{2\alpha}+\cdots+q^{n\alpha}+a_n).
\]

Let us now estimate the integral I. This becomes straightforward when we recall that $z,w\in B$ and take into account the definition of $\Omega$. Indeed,
\[
\text{I}
= \int\limits_{\pi_n(\Omega)}\sum_{m=1}^{n} D^\alpha(X_m^w(\bi),X_m^z(\bi))\,d\mu^n(\bi)
\leq q^{\alpha}+q^{2\alpha}+\cdots+q^{n\alpha}.
\]
Consequently, we have for $n\in\bN$ and every $z,w\in B$
\[    \sum_{k=1}^{n} \int D^{\alpha}(X_k^w(\bi),X_k^z(\bi))\,d\mu^\bN(\bi)
    \leq 2(q^{\alpha}+q^{2\alpha}+\cdots+q^{n\alpha})+a_n(1-\delta).
\]
This proves the assertion.
\end{proof}

By Proposition \ref{proclaim:conunifweak}, there exists $\kappa$ (and hence for all bigger indices) for all $x,y$ we have
\begin{equation}\label{canto}
	\mu^\bN(\Delta_\kappa^{x,y})\geq\frac12\nu(B),\quad\text{ where }\quad
	\Delta_\kappa^{x,y}
	\eqdef\{\bi\in\cI^\bN\colon X_\kappa^x(\bi),X_\kappa^y(\bi)\in B\}.
\end{equation}
Note again that $\bi\mapsto X_\kappa^w(\bi)$ only depends on the first $\kappa$ entries. Let us write $\Delta_\kappa^{x,y}=\pi_k(\Delta_\kappa^{x,y})\times\cI^\bN$. As $\mu^{\bN}$ is a product measure, by Fubini's theorem, for any $n>\kappa$ and every $x,y\in M$ we obtain 
\[\begin{split}
    \sum_{k=\kappa}^{n}& \int D(X_k^x(\bi),X_k^y(\bi))\,d\mu^\bN(\bi)\\
    &=\sum_{k=\kappa}^{n}\,\, \int\limits_{\cI^\bN}\int\limits_{\pi_\kappa(\Delta_\kappa^{x,y})} 
    	D(X_k^x(i_1,\ldots,i_\kappa\bj),X_k^y(i_1,\ldots,i_\kappa\bj))\,d\mu^\kappa(i_1,\ldots,i_\kappa)d\mu^\bN(\bj)\\
&=\int\limits_{\pi_\kappa(\Delta_\kappa^{x,y})}\int\limits_{\cI^{\bN}}
	\sum_{\ell=1}^{n-\kappa}
	D(X_\ell^{X_\kappa^x(i_1,\ldots,i_\kappa)}(\bj),X_\ell^{X_\kappa^y(i_1,\ldots,i_\kappa)}(\bj))\,
		d\mu^{\bN}(\bj)d\mu^\kappa(i_1,\ldots,i_\kappa)\\
&\phantom{=}		
	+	\int\limits_{\cI^\kappa\setminus\pi_\kappa(\Delta_\kappa^{x,y})}\int\limits_{\cI^{\bN}}
	\sum_{\ell=1}^{n-\kappa}
	D(X_\ell^{X_\kappa^x(i_1,\ldots,i_\kappa)}(\bj),X_\ell^{X_\kappa^y(i_1,\ldots,i_\kappa)}(\bj))\,
		d\mu^{\bN}(\bj)d\mu^\kappa(i_1,\ldots,i_\kappa)
\end{split}\]
Let us now apply Claim \ref{someclaim} to $w=X_\kappa^x(i_1,\ldots,i_\kappa)$ and $z=X_\kappa^y(i_1,\ldots,i_\kappa)$. For that recall that if $(i_1,\ldots,i_\kappa)\in\pi_\kappa(\Delta_\kappa^{x,y})$ then $w,z\in B$. Hence, recalling also \eqref{def:an}, we get
\[\begin{split}
	\sum_{k=\kappa}^{n}& \int D(X_k^x(\bi),X_k^y(\bi))\,d\mu^\bN(\bi)\\
	&\leq \mu^\kappa(\pi_\kappa(\Delta_\kappa^{x,y}))\left(2(q^{\alpha}+q^{2\alpha}+\cdots+q^{(n-\kappa)\alpha})
	+a_{n-\kappa}(1-\delta)\right)\\
	&\phantom{\le}
	+a_{n-\kappa}\mu^\kappa(\cI^\kappa\setminus\pi_\kappa(\Delta_\kappa^{x,y}))\\
	&\le 2(q^{\alpha}+q^{2\alpha}+\cdots+q^{n\alpha})
		+ a_n\big((1-\delta)\mu^\kappa(\pi_\kappa(\Delta_\kappa^{x,y})\big)
		+ a_n\big(1-\mu^\kappa(\pi_\kappa(\Delta_\kappa^{x,y}))\big)\\
	&\le 	2(q^{\alpha}+q^{2\alpha}+\cdots+q^{n\alpha})
	+a_n (1-\delta \frac12\nu(B)),
\end{split}\]
where we also used  \eqref{canto}.

As $x,y\in M$ were arbitrary, we get
\[
a_n\leq 2(q^{\alpha}+q^{2\alpha}+\cdots+q^{n\alpha})+a_n (1-\delta \frac12\nu(B))+a_{\kappa},
\]
that is,
\[
	a_n
	\leq \frac{4}{\delta\nu(B)}(q^{\alpha}+q^{2\alpha}+\cdots+q^{n\alpha})+\frac{2a_{\kappa}}{\delta\nu(B)},
\]
for all $n>\kappa$. This implies the assertion.
\end{proof}

%------------------------------------------------------------------------------------------------------
\section{Limit laws: Proof of Theorem \ref{teo3.1-CLTbisneu}}\label{sec33333}
%------------------------------------------------------------------------------------------------------

In this section, assume that $(\cF,\mu)$ is an RDS of continuous maps on a compact metric space $(M,D)$ satisfying \textbf{(LC)} and \textbf{(P)}. We establish the limit laws SLLN, CLT, and LIL for the Markov chain $(X_n^x)_{n\geq 0}$ defined as in \eqref{chain:PropEst}, started at any point $x\in M$.

First note that independent random applications of maps $f\in\cF$ from the RDS $(\cF,\mu)$ gives rise to the Markov chain $(X_n^x)_{n\geq 0}$ with \emph{transition operator}
\[
	\bP(x,A)
	\eqdef \int \delta_{f(x)}(A)\,d\mu(f).
\]
Just note that for the Markov operator defined as in \eqref{cFast}, for any Borel probability measure $\nu$ it holds
\[
	\bM\nu(A)
	= \int\bP(x,A)\,d\bM\nu(x).
\]
For every bounded measurable function $h\colon M\to\bR$  define its image $Ph$ by the \emph{transfer operator},
\[
Ph(x)\eqdef \int h(y) \,\bP(x,dy).
\]
For every $x\in M$,
\[
Ph(x)=\int h(f(x))\,d\mu(f),
\]
and for $n\in\bN$
\begin{equation}\label{using}
P^nh(x)=\int h(f_{\bi}^n(x))\,d\mu^{\bN}(\bi).
\end{equation}
Moreover, if $h$ is a continuous function, $Ph$ is also a continuous function on $M$.%
\footnote{This property is also called weak Feller property.}
Notice that for the $\mu$-stationary measure $\nu$, for any $n\in\bN$ we have
\begin{equation}\label{stattt}
	\nu(h)
	= \int h\,d\nu
	= \int Ph\,d\nu
	= \int P^nh\,d\nu.
\end{equation}

To prove Theorem \ref{teo3.1-CLTbisneu}, fix some H\"older continuous function $h\colon M\to\bR$ and $\alpha\in (0,1]$ and $H>0$ satisfying
\[
|h(x)-h(y)|\leq  H D^{\alpha}(x,y),\quad\mbox{for all}\quad x,y\in M.
\]

%%%%%%%%%%%%%%%%%%%%%%%%%%%%%%%%%%%%%%%%%%
\subsection{CLT}
%%%%%%%%%%%%%%%%%%%%%%%%%%%%%%%%%%%%%%%%%%

By Theorem \ref{teo:sync-on-average}, there is $C'>0$ such that for every $x,y\in M $ and for all $n\in\bN$
\[
	\sum_{k=1}^n\int D^{\alpha}(X_k^x(\bi),X_k^y(\bi))\,d\mu^{\bN}(\bi)
	\leq C'\sum_{k=1}^n q^{k\alpha},
\]
which implies
\begin{equation}\label{eq:CLT-11}
	\sum_{k=1}^n\int |h(f_\bi^k(x))-h(f_\bi^k(y))|\,d\mu^\bN(\bi)
	\le HC'\sum_{k=1}^nq^{k\alpha}
\end{equation}
Setting $C=HC'\sum_{k=1}^{\infty}\lambda^{k\alpha}$, it follows from \eqref{using} and the above that
\[\begin{split}
 	\Big| \sum_{k=1}^{n}P^kh(x)-\sum_{k=1}^{n}P^kh(y)\Big|
	&\le \sum_{k=1}^n\Big| \int \big(h(f_\bi^k(x))-h(f_\bi^k(y))\big)\,d\mu^\bN(\bi)\Big|\\
	&\le \sum_{k=1}^n\int |h(f_\bi^k(x))-h(f_\bi^k(y))|\,d\mu^\bN(\bi)\\
	&\le HC'\sum_{k=1}^{n}q^{k\alpha}\leq HC'\sum_{k=1}^{\infty}q^{k\alpha}
	= C.
\end{split}\]
 Therefore, using also \eqref{stattt}, we get that
\begin{align}\begin{split}\label{bound:operatorL2}
    \Big| \sum_{k=1}^{n}P^kh(x)-n\nu(h)\Big|
    &=  \Big| \sum_{k=1}^{n}P^kh(x)-\int\sum_{k=1}^{n}P^kh(y)\,d\nu(y)\Big|\\
    &\leq \int \Big| \sum_{k=1}^{n}P^kh(x)-\sum_{k=1}^{n}P^kh(y)\Big|\,d\nu(y)\\
    &\leq C.
\end{split}\end{align}
Therefore, by \cite{DerrLin:2003}, the limit $\sigma^2(h)$ defined in \eqref{defsigma2} exists and is finite. Moreover, for $\nu$-almost all $x$
    \[
    \frac{1}{\sqrt{n}}\sum_{k=0}^{n-1} h\left(X_k^x\right)
    \to \mathcal N(\nu(h),\sigma^2(h))
        \quad\text{ as }n\to\infty.
	\]
Equivalently, by L\'evy's continuity theorem, we can write
\begin{equation}\label{conv:ch.f}
	\lim_{n\to\infty}
	\int\exp\Big(it\frac{h(X_n^x(\bi))+\cdots+h(X_1^x(\bi))}{\sqrt{n}}\Big)\,d\mu^{\bN}(\bi)
	=\exp\Big(it\nu(h)-\frac{1}{2}t^2\sigma^2(h)\Big).
\end{equation}

Let us argue that \eqref{conv:ch.f} holds \emph{for any} $x\in M $.
Fix some $y\in\supp(\nu)$ such that \eqref{conv:ch.f} holds. Given $n\in\bN$, together with \eqref{eq:CLT-11} we have
\[\begin{split}
    &\Big| \hspace{-0.1cm}\int\hspace{-0.1cm}
     \Big(\hspace{-0.1cm}\exp\hspace{-0.05cm}
     	\Big(\hspace{-0.05cm}it\frac{h(X_n^x(\bi))+\cdots+h(X_1^x(\bi))}{\sqrt{n}}\Big)
    		\hspace{-0.1cm}-\exp\hspace{-0.05cm}
	\Big(\hspace{-0.05cm}it\frac{h(X_n^y(\bi))+\cdots+h(X_1^y(\bi))}{\sqrt{n}}\Big)\Big)
		\,d\mu^{\bN}(\bi)\Big|\\
    &\leq\frac{|t|}{\sqrt{n}}\int \sum_{k=1}^{n}\left|h(X_k^x(\bi))-h(X_k^y(\bi)) \right| \,d\mu^{\bN}(\bi)\\
    &=\frac{|t|}{\sqrt{n}} \sum_{k=1}^{n}\int\left|h(f_\bi^n(x))-h(f_\bi^k(k)) \right| \,d\mu^{\bN}(\bi)\\
    &\leq\frac{|t|}{\sqrt{n}}HC' \sum_{k=1}^{n} q^{k\alpha}
    \leq\frac{|t|}{\sqrt{n}}HC' \sum_{k=1}^\infty q^{k\alpha}
    = \frac{|t|}{\sqrt n}C.
\end{split}\]
Since the latter term converges to 0 as $n\to\infty$, we conclude that \eqref{conv:ch.f} holds for every $x\in M $. Therefore, the CLT holds for the Markov chain $(X_n^x)_{n\ge0}$.

%%%%%%%%%%%%%%%%%%%%%%%%%%%%%%%%%%%%%%%%%%
\subsection{LIL}
%%%%%%%%%%%%%%%%%%%%%%%%%%%%%%%%%%%%%%%%%%

We assume, without loss of generality, that $\nu(h)=0$.
Note that the Markov chain $(X_n)_{n\geq 0}$ is a stationary stochastic sequence with initial distribution $\nu$. By Kolmogorov extension theorem, $(X_n)_{n\geq 0}$ can be embedded in a two-sided stationary stochastic sequence $(Y_n)_{n\in\bZ}$ on some
probability space $(\Omega,\mathcal{A},\prob)$ with law $\nu$. For every $m\in\bZ$ and $n\geq0$ the random vectors $(Y_m,\ldots,Y_{m+n})$ and $(X_0,\ldots,X_n)$ are equal in distribution.

 Let $\mathcal{A}_0$ be the sigma-algebra generated by $(Y_n)_{n\leq 0}$. For $Y$ a random variable on $(\Omega,\mathcal{A},\prob)$, define $\mathbb{E}(Y)\eqdef \int Y d\prob$ and let $\mathbb{E}(Y|\mathcal{A}_0)$ be the conditional expectation of $Y$ given $\mathcal{A}_0$. Hence,
\begin{align*}
    \lVert\mathbb{E}(h(Y_1)+\cdots+h(Y_n)|\mathcal{A}_0)\rVert^2&=\int_{ M }\lvert \mathbb{E}(h(Y_1)+\cdots+h(Y_n)|Y_0=x)\rvert^2d\nu(x)\\
    &=\int_{ M }\lvert Ph(x)+\cdots+P^nhx)\rvert^2d\nu(x),
\end{align*}
where $\lVert\cdot\rVert $ denotes the norm in $L^2(\nu)$.
By \eqref{bound:operatorL2},
\begin{align*}
     \lVert\mathbb{E}(h(Y_1)+\cdots+h(Y_n)|\mathcal{A}_0)\rVert^2
    \leq C^2<\infty.
\end{align*}
Therefore,
\[
\sum_{n=1}^{\infty}\frac{\log n}{n^{3/2}}\lVert\mathbb{E}(h(Y_1)+\cdots+h(Y_n)|\mathcal{A}_0)\rVert<\infty.
\]
By \cite[Corollary 1]{ZhaoWood:2008}, $\prob$-almost surely
\[
\limsup_{n\to\infty}\frac{h(Y_1)+\cdots+h(Y_n)}{\sqrt{2n\log \log (n)}}=t,
\]
where $t^2(h)\eqdef\lim_{n\to\infty}\frac{1}{n}\mathbb{E}(h(Y_1)+\cdots+h(Y_n))^2$. Since $(Y_n)_{n\geq 0}$ and $(X_n)_{n\geq 0}$ have the same law,
we obtain that $t^2(h)=\sigma^2(h)$. Moreover, for $(\mu^{\bN}\otimes\nu)$-almost every $(\bi,y)$
\begin{align}\label{lil:onsupp}
\limsup_{n\to\infty}\frac{h(X^y_1(\bi))+\cdots+h(X^y_n(\bi))}{\sqrt{2n\log \log (n)}}=\sigma.
\end{align}
To conclude that \eqref{lil:onsupp} holds for every $x$ and $\mu^{\bN}$-almost every $\bi$, recall now that by Proposition \ref{prop:sync-exp} we have exponential synchronization. Given any $x\in M$, take $y\in\supp(\nu)$ such that \eqref{lil:onsupp} holds for $\mu^{\bN}$-almost every $\bi$ and apply Proposition \ref{prop:sync-exp} to $x,y$ to conclude.

%%%%%%%%%%%%%%%%%%%%%%%%%%%%%%%%%%%%%%%%%%
\subsection{SLLN}
%%%%%%%%%%%%%%%%%%%%%%%%%%%%%%%%%%%%%%%%%%

Similarily to the proof of Proposition \ref{prop:sync-exp}, let us consider the auxiliary set $\mathcal{E}'$ of points $(\bi,x)\in \cI^{\bN}\times M$ such that the two following conditions hold:
\begin{itemize}
    \item[(a')] $\lim_{n\to\infty}\frac{1}{n}\sum_{k=0}^{n-1} h\left(X_k^x(\bi)\right)=\nu(h)$,
    \item[(b')] there exists a neighborhood $B$ of $x$ such that for all $n\in\bN$, $\mbox{diam}_D\left( f_\bi^n(B)\right)
	\le q^n$.
\end{itemize}
Note that $\mathcal{E}'$ is invariant by the skew-product $T$. Now, let us show that $\mathcal{E}'$ has full $\mu^{\bN}\otimes\nu$-measure.
By Birkhoff's Ergodic Theorem,
\[
	\left(\mu^{\bN}\otimes\nu\right)
		\Big(\Big\{(\bi,x)\in \cI^{\bN}\times M
	\colon\mbox{ ($a'$) holds}\Big\}\Big)=1.
\]
Since $\mu^{\bN}\otimes\nu=\int_{ M }\mu^{\bN}\otimes\delta_x \,d\nu(x)$, we can apply \textbf{(LC)}  to get
\[
\left(\mu^{\bN}\otimes\nu\right)\left(\left\{(\bi,x)\in \cI^{\bN}\times M\colon \mbox{ (b') holds}\right\}\right)=1.
\]
Therefore,
$
\left(\mu^{\bN}\otimes\nu\right)\left(\mathcal{E}'\right)=1.
$

Note that \[\mathcal{E}'=\bigcup_{\bi\in\cI^{\bN}}\{\bi\}\times U(\bi),\] where $U(\bi)$ is the set of all points $x\in M $ such that (a') and (b') hold. Let us prove that for every $\bi$ the set $U(\bi)$ is open in $ M $. Indeed, given $\bi\in\cI^{\bN}$, either there is no $x\in U(\bi)$ such that (b') holds and hence $U(\bi)$ is empty and, in particular, open. Or, $U(\bi)$ is nonempty and for every $x\in U(\bi)$ there exists an open neighborhood $B$ of $x$ such that (b') holds. This together with the fact that (a') holds for $x$, implies
\[
 \lim_{n\to\infty}\frac{1}{n}\sum_{k=0}^{n-1} h\left(X_k^z(\bi)\right)=\nu(h),
\]
for every $z\in B$. Taking the same neighborhood $B$ for every $z\in B$, it follows that (b') holds for $(\bi,z)$ and we conclude that $B\subset U(\bi)$. Consequently, $U(\bi)$ is open in $ M $.

Using uniqueness of $\nu$ and applying \cite[Proposition 4.2]{Mal:17} to the set $\mathcal{E}'$,  we get for every $x\in M$
\[
\left(\mu^{\bN}\otimes\delta_x\right)\left(\mathcal{E}'\right)=1.
\]
This completes the proof of SLLN.
\qed

%-----------------------------------------
\section{Large deviations of synchronization rates and Lyapunov exponents: IFSs on $\bS^1$ }\label{sec:proooof}
%-----------------------------------------

Through this section, we consider the particular case of $(\cF,\mu)$ being an IFS with probabilities on $\bS^1$. More precisely, consider a finite collection of circle $C^1$ diffeomorphisms $\cF=\{f_0,\ldots,f_{N-1}\}$ and we assume that \textbf{(H)} and \textbf{(P)} are satisfied. Let $\cI=\{0,\ldots,N-1\}$ and consider the Bernoulli probability measure $\mu^{\bN}$ induced by a probability vector $(p_0,\ldots,p_{N-1})$, where $\mu(\{i\})=p_i$ for $i\in\{0,\ldots,N-1\}$. We will always assume that $\mu$ is non-degenerate, that is,  $p_i>0$ for all $i$.
We will consider the standard distance $d$ defined in \eqref{def:metusualcircle}.

In Section \ref{secprooof} we first prove Theorem \ref{teo3.1-CLT} invoking the technique of metric change to prove contraction on average, which puts us into a position to derive the claimed limit laws. In Section \ref{secMarSys} we introduce Markov systems. In Section \ref{seclargedesvios}, we will additionally assume that $\cF$ are $C^{1+\beta}$ diffeomorphisms and apply the preliminary results from Section \ref{secMarSys} to the potential which governs the Lyapunov exponents and prove Theorem \ref{teolargedesvios}.

%-----------------------------------------
\subsection{Limit laws: Proof of Theorem \ref{teo3.1-CLT}}\label{secprooof}
%-----------------------------------------

By \cite[Theorem 1.2]{GelSal:}, there exist $\alpha\in(0,1)$ and $\lambda\in(0,1)$ such that for every $x,y\in\bS^{1}$
\begin{equation}\label{def:metricCA}
 	\int d^{\alpha}(X_k^x(\bi),X_k^y(\bi))\,d\mu^{\bN}(\bi)\leq \lambda d^{\alpha}(x,y).
\end{equation}
Note that $d^{\alpha}$ is a is equivalent metric to $d$ on $\bS^1$ and has the property \textbf{(CA)}, that is, $(\cF,\mu)$ is contracting on average on $(\bS^1,d^{\alpha})$. Hence, by \cite[Theorem 2.1]{BDEG88}, there is a unique stationary probability, proving the first assertion.

On the other hand, contraction on average implies that, in particular, \cite[condition (H3)]{Pei:93} is satisfied for $k_0=k$ (note that the remaining conditions \cite[(H0)--(H2)]{Pei:93} are an immediate consequence of our hypotheses). Further, as explained above, hypotheses (H1) and (H2) imply synchronization \textbf{(S)} and hence \cite[condition (H4)]{Pei:93} is satisfied. Hence, all hypotheses of \cite[Theorem 5.1]{Pei:93} are satisfied and the SLLN holds for $(\cF,\mu)$.

Now, let us prove CLT. Let $h\colon\bS^1\to\bR$ be a H\"older continuous function. Take $\beta\in(0,1]$ and $C>0$ such that for every $x,y\in\bS^1$
\[
|h(x)-h(y)|\leq C d^{\beta}(x,y).
\]
Since the metric satisfies $d\leq 1$, for $\alpha\in(0,1)$ as in \eqref{def:metricCA}, it follows $d^\beta\le d^{\alpha\beta}$ and hence that $h$ is Lipschitz with respect to the metric $d^{\alpha \beta }$. Let us show that \cite[condition (H3)]{Pei:93} for the metric $d^{\alpha \beta}$ and $k_0=k$ is satisfied. Indeed, by Jensen's inequality,
\[
\int d^{\alpha \beta}(X_k^x(\bi),X_k^y(\bi))\,d\mu^\bN(\bi)\leq\left(\int d^{\alpha}(X_k^x(\bi),X_k^y(\bi))\,d\mu^{\bN}(\bi)\right)^{\beta}\leq \lambda^{\beta} d^{\alpha\beta}(x,y),
\]
which implies the desired.
As $d$ and $d^{\alpha \beta}$ are equivalent metrics, we get \cite[(H0)--(H2), (H4)]{Pei:93} hold. Again, all hypotheses of \cite[Theorem 5.1]{Pei:93} for the metric $d^{\alpha \beta}$ are satisfied and CLT holds true for $(\cF,\mu)$.

This proves the theorem.\qed

%-----------------------------------------
\subsection{Markov systems}\label{secMarSys}
%-----------------------------------------

Denote by $\nu$ the stationary measure for $(\cF,\mu)$ (by Theorem \ref{teo3.1-CLT}, it is unique). Let $L\geq 1$ be such that for all symbols $j\in \cI=\{0,\ldots,N-1\}$ and every $x,y\in\bS^1$,
\begin{equation}\label{maxLip}
    L^{-1}d(x,y)
    \leq d(f_j(x),f_j(y))\leq Ld(x,y),
\end{equation}
Below we will invoke Theorem \ref{teo:sync-on-average}; let $\alpha\in(0,1)$, $\lambda\in(0,1)$, and $c>0$ as provided by it. Consider the space
\begin{equation}\label{eq:defnormalpha}\begin{split}
	&\mathcal{H}_{\alpha}
	= \mathcal{H}_{\alpha}(\cI\times\bS^1)
	\eqdef \big\{\varphi\colon \cI\times\mathbb{S}^1\rightarrow \mathbb{R}\colon
		\varphi\text{ bounded measurable},  \|\varphi\|_{\alpha}<\infty\big\},\\
	&\quad\text{ where }\quad	
	\|\varphi\|_{\alpha}\eqdef \|\varphi\|_{\infty}
	+ \lvert\varphi\rvert_\alpha
	\quad\text{ with }\quad
	\lvert\varphi\rvert_\alpha
	\eqdef\sup_{j ,x\neq y}\frac{|\varphi(j,x)-\varphi(j,y)|}{d^{\alpha}(x,y)}.
\end{split}	\end{equation}
Note that $(\mathcal{H}_{\alpha},\|\varphi\|_{\alpha})$ is a Banach algebra with unity.

Consider the Markov system $(K,\mu\otimes\nu)$, where the Markov kernel $K$ is defined by
\begin{equation}\label{defMarkovs}
K\colon \cI\times\bS^1\rightarrow\mbox{Prob}(\cI\times\bS^1),\quad	
K(j,x)
	\eqdef\int\delta_{\left(i,f_{j}(x)\right)}\,d\mu(i)=\sum_{i=1}^{N-1}p_i\delta_{\left(i,f_{j}(x)\right)},
\end{equation}
where $\delta_{(\cdot)}$ denotes the usual Dirac measure supported on $\cI\times\bS^1$.
Note that $K$ determines the Laplace-Markov operator $\mathcal{Q}$ on the space $L^{\infty}(\cI\times \mathbb{S}^1)$ defined by
\[
	\mathcal{Q}(\varphi)(j,x)
	\eqdef\int\varphi(i,f_{j}(x))\,d\mu(i)
	=\sum_{i=1}^{N-1}p_i\varphi(i,f_{j}(x)).
\]

The following lemma establishes a relation between this operator and the skew-product  $T$ defined in \eqref{def:skewprod-RD}. Let $\pi\colon\cI^{\bN}\times \mathbb{S}^1\to \cI\times \mathbb{S}^1$ be the projection defined by $\pi(\bi,x)\eqdef(i_1,x)$.

\begin{lemma}\label{lemind}
For every $\varphi\in L^{\infty}(\cI\times\mathbb{S}^1)$, $j\in\cI$, $x\in\bS^1$, and $n\in\bN$,
\[
    (\mathcal{Q}^{n}(\varphi))(j,x)
    	=\int\big(\varphi\circ\pi\circ T^{n-1}\big)(\bi,f_{j}(x))\,d\mu^{\bN}(\bi).
\]
\end{lemma}

\begin{proof}
We argue by induction on $n$. The case $n = 1$ is clear. Suppose that the assertion holds for $n$ and let us prove it for $n+1$.
Indeed, using the inductive hypothesis we get
\[\begin{split}
    (\mathcal{Q}^{n+1}(\varphi))(j,x)
    &=\int \mathcal{Q}^n(\varphi)(i_1,f_{j}(x))\,d\mu(i_1)\\
    \small{\text{(by induction hypothesis) }}\quad
    &=\int \int\big(\varphi\circ\pi\circ T^{n-1}\big)(\bj,f_{i_1}(f_{j}(x)))
    	\,d\mu^{\bN}(\bj)\,d\mu(i_1)\\
    &= \sum_{i=0}^{N-1} p_i\int\big(\varphi\circ\pi\circ T^{n-1}\big)(\bj,f_i(f_{j}(x)))\,d\mu^{\bN}(\bj)\\
    \small{\text{(by $\sigma$-invariance of $\mu^\bN$) }}\quad
    &= \sum_{i=0}^{N-1} p_i\int\big(\varphi\circ\pi\circ T^{n-1}\big)
    	(\sigma(\bj),f_i(f_{j}(x)))\,d\mu^{\bN}(\sigma(\bj)).
\end{split}\]
Note that $\sigma(\bj)=\sigma(j,j_2,j_3,\ldots,j_n,\ldots)$ for every $j\in\cI$. Hence,
\[\begin{split}
    (\mathcal{Q}^{n+1}(\varphi))(j,x)
    &= \sum_{j_1=0}^{N-1} p_{j_1}\int
    	\big(\varphi\circ\pi\circ T^{n-1}\big)(\sigma(\bj),f_{j_1}(f_j(x)))\,d\mu^{\bN}(\sigma(\bj))\\
    &=\int\big(\varphi\circ\pi\circ T^{n-1}\big)
    	(\sigma(\bj),f_{j_1}(f_j(x)))\,d\mu^{\bN}(\bj)\\
    &=\int\big(\varphi\circ\pi\circ T^n\big)(\bj,f_j(x))\,d\mu^{\bN}(\bj),
\end{split}\]
which proves the assertion for $n+1$ and concludes the proof.
\end{proof}

\begin{proposition}\label{prop2}
The Laplace-Markov operator $\mathcal{Q}$ acts simply and quasi-compactly on $\mathcal{H}_{\alpha}(\cI\times\mathbb{S}^1)$ relative to the   measure $\mu\otimes\nu$.
\end{proposition}

\begin{proof}
To prove the assertion, we need to show that there are $C>0$ and $\gamma\in(0,1)$ such that for every $\varphi\in\mathcal{H}_{\alpha}$ and $n\in\bN$,
\[
	\lVert \mathcal{Q}^n\varphi - \big(\int\varphi\,d(\mu\otimes\nu)\big)\mathbf{1}\rVert_\alpha	
	\le C\gamma^n\lVert\varphi\rVert_\alpha.
\]
This will be done in two steps (Claims \ref{claim01} and \ref{claim02}). Recall the definition of $\lvert\cdot\rvert_\alpha$ in \eqref{eq:defnormalpha}.

\begin{claim}\label{claim01}
There is $C>0$ so that for all $\varphi\in L^{\infty}(\cI\times\mathbb{S}^1)$ and $n\in \mathbb{N}$
\[
	|\mathcal{Q}^{n}\varphi|_{\alpha}
	\leq C\lambda^n |\varphi|_{\alpha}.
\]
\end{claim}

\begin{proof}
Recall that $\pi(\bi,x)\eqdef (i_1,x)$. By Lemma \ref{lemind}, together with the definition of $\lvert\cdot\rvert_\alpha$ in \eqref{eq:defnormalpha}, we get
\[\begin{split}
     &|\mathcal{Q}^{n}\varphi|_{\alpha}
     =\sup_{j,x\neq y}\Big|\int\frac{\left(\varphi\circ\pi\circ T^{n-1}(\bi,f_{j}(x))
     		-\varphi\circ\pi\circ T^{n-1}(\bi,f_{j}(y))\right)}{d^{\alpha}(x,y)}\,d\mu^{\bN}(\bi)\Big|\\
     &\leq\sup_{j,x\neq y}\int
     	\frac{\big|\varphi\circ\pi(\sigma^{n-1}(\bi),f_{\bi}^{n-1}\circ f_{j}(x))-\varphi\circ\pi(\sigma^{n-1}(\bi),f_{\bi}^{n-1}\circ f_{j}(y))\big|}{d^{\alpha}(x,y)}\,d\mu^{\bN}(\bi),\\
     &\le |\varphi|_{\alpha}\sup_{j,x\neq y}
     	\int\frac{d^{\alpha}(f_{\bi}^{n-1}\circ f_{j}(x),f_{\bi}^{n-1}\circ f_{j}(y))}
		{d^{\alpha}\big( f_{j}(x), f_{j}(y) \big)}\frac{d^{\alpha}\big( f_{j}(x), f_{j}(y) \big)}{d^{\alpha}(x,y)}\,d\mu^{\bN}(\bi)
\end{split}\]
Hence, from \eqref{maxLip}  we get		
\[\begin{split}
   |\mathcal{Q}^{n}\varphi|_{\alpha}
    & \leq L^\alpha |\varphi|_{\alpha}\sup_{j,x\neq y}
    	\int\frac{d^{\alpha}\big(f_{\bi}^{n-1}\circ f_{j}(x),f_{\bi}^{n-1}\circ f_{j}(y) \big)}{d^{\alpha}\big( f_{j}(x), f_{j}(y) \big)}\,d\mu^{\bN}(\bi)\\
    &= L^\alpha\lvert\varphi\rvert_\alpha\sup_{x\ne y}\int\frac{d^{\alpha}(X_{n-1}^x(\bi),X_{n-1}^y(\bi))}{d^\alpha(x,y)}\,d\mu^{\bN}(\bi)\\
    &\le L^\alpha |\varphi|_{\alpha}\cdot c\lambda^{n-1},
\end{split}\]
where for the last inequality we applied \eqref{def:metricCA}. Letting $C\eqdef L^{\alpha}c\lambda^{-1}$, this proves the claim.
\end{proof}

\begin{claim}\label{claim02}
	There is $C>0$ so that for all $\varphi\in L^{\infty}( \cI\times\mathbb{S}^1)$ and $n\in \mathbb{N}$
\[
	\Big\|\mathcal{Q}^{n}\varphi-\Big(\int\varphi\,d(\mu\otimes\nu)\Big)\mathbf{1}\Big\|_{\infty}
\leq C\lambda^n |\varphi|_{\alpha}.
\]
\end{claim}

\begin{proof}
As $\nu$ is stationary, the measure $\mu^{\bN}\otimes\nu$ is $T$-invariant. Thus,
\[
\int(\varphi\circ\pi)\,d(\mu^{\bN}\otimes\nu)=\int(\varphi\circ \pi\circ T^{n-1})\,d(\mu^{\bN}\otimes\nu).
\]
Hence, using Lemma \ref{lemind},
\begin{align*}
    &\Big\|\mathcal{Q}^{n}\varphi-\Big(\int\varphi\,d(\mu\otimes\nu)\Big)\mathbf{1}\Big\|_{\infty}
    = \sup_{j\in\cI,y\in\mathbb{S}^1}\Big| \mathcal{Q}^{n}\varphi(j,y)-\int(\varphi\circ\pi)\,d(\mu^{\bN}\otimes\nu)\Big|\\
    &= \sup_{j\in\cI,y\in\mathbb{S}^1}\Big| \int(\varphi\circ\pi\circ T^{n-1})(\bi,f_j(y))\,d\mu^{\bN}(\bi)-
    	\int(\varphi\circ \pi\circ T^{n-1})\,d(\mu^{\bN}\otimes\nu)\Big|\\
&\leq\sup_{j,y} \int \int\left|\varphi\circ \pi\circ T^{n-1}(\bi,f_j(y))-\varphi\circ \pi\circ T^{n-1}(\bi,x)  \right| \,d\mu^{\bN}(\bi) d\nu(x)\\	
&\leq\sup_{j,y} \int \int\left|\varphi\circ \pi \big(\sigma^{n-1}(\bi),f_\bi^{n-1}f_j(y)\big)-\varphi\circ \pi\big(\sigma^{n-1}(\bi),f_\bi^{n-1}(x)\big)  \right| \,d\mu^{\bN}(\bi) d\nu(x).
\end{align*}
By definition of $|\varphi|_{\alpha}$ in \eqref{eq:defnormalpha} and Theorem \ref{teo:sync-on-average} we get
\begin{align*}
    \Big\|\mathcal{Q}^{n}\varphi-\Big(\int\varphi\,d(\mu\otimes\nu)\Big)&\mathbf{1}\Big\|_{\infty}\\
    &\leq |\varphi|_{\alpha} \sup_{j,y} \int\int d^{\alpha}\big(f^{n-1}_{\bi}( f_j(y)),f^{n-1}_{\bi}(x)\big)\, d\mu^{\bN}(\bi)d\nu(x)\\
    &\leq c\lambda^{n-1}|\varphi|_{\alpha} \sup_{j,y} \int d^{\alpha}\left( f_j(y),x\right)\, d\nu(x)
    = c\lambda^{n-1}|\varphi|_{\alpha} .
\end{align*}
Letting $C\eqdef c\lambda^{-1}$ proves the claim.
\end{proof}

Recalling the definition of $\lVert\cdot\rVert_\alpha$ in \eqref{eq:defnormalpha}, it follows from  Claim \ref{claim01} and Claim \ref{claim02} that there exist $C>0$ so that
\[
	\lVert \mathcal{Q}^n\varphi - \big(\int\varphi\,d(\mu\otimes\nu)\big)\mathbf{1}\rVert_\alpha	
	\le 2C\lambda^n\lvert\varphi\rvert_\alpha.	
\]
This proves the proposition.
\end{proof}

%-----------------------------------------
\subsection{Large deviations of Lyapunov exponents}\label{seclargedesvios}
%-----------------------------------------
In this section, we assume additionally that $\cF$ is an IFS of $C^{1+\beta}$ diffeomorphisms of $\bS^1$, $\beta>0$, satisfying \textbf{(LC)} and \textbf{(P)}.

Consider the function $\phi\colon \cI\times\mathbb{S}^1\rightarrow \mathbb{R}$ given by
\[
	\phi(j,x)\eqdef \log|(f_{j})'(x)|.
\]
By \textbf{(P)}, $\mathcal{F}$ is proximal on $(\bS^1,d)$. Hence, by Claim \ref{claim:omegafullmeas}, for $q\in(0,1)$ so that \textbf{(LC)} holds, we have the following (exponential) synchronization:
for $x,y\in\bS^1$, $x\ne y$, the set
\[
     \Omega^{x,y}
     \eqdef \big\{\bi\in\cI^{\bN}\colon
      \, \exists C>0, 	\mbox{ such that }
		d(X_n^x(\bi),X_n^y(\bi))\leq C q^n \text{ for all } n\in\bN\big\}.
\]
has full measure:
 \[
	\mu^{\bN}(\Omega^{x,y})=1.
\]

Before proving Theorem \ref{teolargedesvios} let us first establish some distortion results.
For $\bi\in\Omega^{x,y}$ define
\[
	I_{\bi}
	\eqdef \begin{cases} [x,y],  \text{ if } \lim_{n\to\infty}\big|f_{\bi}^n([x,y])\big|=0,\\
					[y,x], \text{ if } \lim_{n\to\infty}\big|f_{\bi}^n([y,x])\big|=0
	\end{cases}
	\quad\text{ and }\quad
		\delta_n(\bi)
	\eqdef |f_{\bi}^n(I_\bi)|.
\]
Consider the following modulus of continuity
\[
	\omega(\delta)
	\eqdef \max_{j\in\cI}\max_{d(z,w)\leq \delta}\left|\log|f_j'(z)|-\log|f_j'(w)|\right|
\]
and note that $\omega(\delta)\to0$ as $\delta\to0$.

\begin{lemma}[Tempered distortion in average]\label{clac2dif}
For every $x,y\in\bS^1$, $x\neq y$, and $\bi\in\Omega^{x,y}$,
\[
	\lim_{n\to\infty}\frac{1}{n}\log\max_{z,w\in I_\bi}\Big|\frac{(f_{\bi}^n)'(z)}{(f_{\bi}^n)'(w)}\Big|
    =0
    \quad\text{ and }\quad
    \lim_{n\to\infty}\frac{1}{n}\int
		\log\max_{z,w\in I_\bi}\Big|\frac{(f_{\bi}^n)'(z)}{(f_{\bi}^n)'(w)}\Big| \,d\mu^{\bN}(\bi)=0.
\]
\end{lemma}

\begin{proof}
We have that
\[
	\log\max_{z,w\in I_\bi}\Big|\frac{(f_{\bi}^n)'(z)}{(f_{\bi}^n)'(w)}\Big|
	\leq \sum_{k=0}^{n-1}\omega({\delta_k(\bi)}),
\]
By the above $\delta_k(\bi)\to0$ and hence $\omega(\delta_k(\bi))\to0$. This implies the first assertion.
Since $\mu^{\bN}(\Omega^{x,y})=1$ and
\[
	0\leq\frac{1}{n}\log\max_{z,w\in I_\bi}\left|\frac{(f_{\bi}^n)'(z)}{(f_{\bi}^n)'(w)}\right|
	\leq2\log L,
\]
the dominated convergence theorem implies the claim.
\end{proof}

We are now ready to show Theorem \ref{teolargedesvios}.

\begin{proof}[Proof of Theorem \ref{teolargedesvios}]
Consider the Markov system $(K,\mu\otimes\nu)$ defined in \eqref{defMarkovs}. By Proposition \ref{prop2}, it acts simply and quasi-compactly on the Banach algebra $\mathcal{H}_\alpha=\mathcal{H}_{\alpha}(\cI\times\mathbb{S}^1)$ relative to $\mu\otimes\nu$. As $\cF$ is assumed to consist of $C^{1+\beta}$ maps, $\phi$ is $\alpha$-H\"older for any $\alpha\in(0,\min\{\alpha_0,\beta\}]$. Hence, $\phi\in \mathcal{H}_{\alpha}$. Hence, we can apply \cite[Theorem 4.4]{DuaKle:17} to the function $\phi$. It guarantees that there are positive constants $h,\varepsilon_0,c$ such that for every $\varepsilon\in(0,\varepsilon_0)$, $(j,x)\in \cI\times \mathbb{S}^1$, and $n\in\bN$,
\[
\mu^\bN\left(\bi\in\cI^{\bN}\colon
	i_1=j \mbox{ and }
	\left|\log |(f^{n}_{\bi})'(x)|-n\gamma(\mu)\right|>n\varepsilon \right)
	\leq ce^{-nh\varepsilon^2 }.
\]

Averaging in $j\in\cI$ with respect to $\mu$, for every $x\in\bS^1$ we get
\begin{equation}\label{LDP1}
	\mu^\bN\left(\bi\in\cI^{\bN}
	\colon\left|\log |(f^{n}_{\bi})'(x)|-n\gamma(\mu)\right|>n\varepsilon \right)
	\leq ce^{-nh\varepsilon^2 }.
\end{equation}
This proves the second assertion of the theorem.

To prove the third assertion in the theorem, first observe that for every $\bi\in\Omega^{x,y}$, by the mean value theorem and triangular inequality, it holds
\[
    \Big|\log \frac{d(X_n^x(\bi),X_n^y(\bi))}{d(x,y)}-n\gamma(\mu)\Big|
    \le \log\max_{z,w\in I_\bi}\Big|\frac{(f_{\bi}^n)'(z)}{(f_{\bi}^n)'(w)}\Big|
    	+|\log |(f_{\bi}^n)'(x)|-n\gamma(\mu)|.
\]
Therefore,
\[\begin{split}
    \mu^\bN&\Big(\Big|\log \frac{d(X_n^x(\bi),X_n^y(\bi))}{d(x,y)}-n\gamma(\mu)\Big|
    	>n\varepsilon\Big)\\
    &\le \mu^\bN\Big( \max_{z,w\in I_\bi}\Big|
    		\frac{(f_{\bi}^n)'(z)}{(f_{\bi}^n)'(w)}\Big|>e^{n\varepsilon/2}\Big)
    +\mu^\bN\Big(\big|\log |(f_{\bi}^n)'(x)|-n\gamma(\mu)\big|>\frac{n\varepsilon}{2}\Big)\\
    &\eqdef I_1(\varepsilon,n) +I_2(\varepsilon,n)
\end{split}\]
It follows from Chebyshev's inequality that
\begin{equation}\label{eq3.5:2}
	I_1(\varepsilon,n)
	\le e^{-n\varepsilon/2}\int\max_{z,w\in I_\bi}\Big|\frac{(f_{\bi}^n)'(z)}{(f_{\bi}^n)'(w)}\Big|\,d\mu^{\bN}(\bi).
\end{equation}
As a consequence of Lemma \ref{clac2dif},  the dominated convergence theorem implies
\[
    \lim_{n\to\infty}\frac{1}{n}\int\max_{z,w\in I_\bi}\Big|\frac{(f_{\bi}^n)'(z)}{(f_{\bi}^n)'(w)}\Big|\,d\mu^{\bN}(\bi)=1.
\]
Hence, for every $x,y$ there exists a constant $C=C(x,y)>0$ such that for all $n\in\bN$
\[
	\int\max_{z,w\in I_\bi}\left|\frac{(f_{\bi}^n)'(z)}{(f_{\bi}^n)'(w)}\right|\,d\mu^{\bN}(\bi)
	\le C n,
\]
and \eqref{eq3.5:2} implies
\begin{equation}\label{I1}
     I_1(\varepsilon,n)
     \le Cn e^{-n\varepsilon/2}.
\end{equation}
To estimate $I_2(\varepsilon,n)$, for all $\varepsilon\in(0,\varepsilon_0)$ the estimate \eqref{LDP1} implies
\begin{equation}\label{I2}
  I_2(\varepsilon,n)
  \le ce^{-nh\varepsilon^2/4}.
\end{equation}
The estimates \eqref{I1} and \eqref{I2} together imply
\[
	\mu^\bN\Big(\Big|\log \frac{d(X_n^x(\bi),X_n^y(\bi))}{d(x,y)}-n\gamma(\mu)\Big|>n\varepsilon\Big)
	\le C n e^{-n\varepsilon/2}+ ce^{-nh\varepsilon^2/4}.
\]
This implies the third assertion of the theorem.

Now, let us prove
\begin{equation}\label{eq:lim-liap-exp-1}
	\lim_{n\to\infty}\frac{1}{n}\int\log|(f_{\bi}^n)'(x)|\,d\mu^{\bN}(\bi)
	=\gamma(\mu).
\end{equation}
For $\varepsilon\in(0,\varepsilon_0)$ define
\[\begin{split}
	\Lambda_{\varepsilon}
	\eqdef\Big\{&\bi\in\cI^{\bN}\colon\\
	&\gamma(\mu)-\varepsilon
	\leq \liminf_{n\to\infty}\frac{1}{n}\log|(f_{\bi}^n)'(x)|\leq\limsup_{n\to\infty}\frac{1}{n}\log|(f_{\bi}^n)'(x)|
	\leq \gamma(\mu) +\varepsilon.
\end{split}\]
Note that \eqref{LDP1} together with the Borel-Cantelli lemma implies $\mu^{\bN}(\Lambda_{\varepsilon})=1$. Hence
\[
	1
	= \mu^\bN\Big(\bigcap_{k=1}^{\infty}\Lambda_{1/k}\Big)
	= \mu^\bN\Big(\bi\in\cI^{\bN}\colon
		\lim_{n\to\infty}\frac{1}{n}\log|(f_{\bi}^n)'(x)|=\gamma(\mu)\Big).
\]
So we conclude that for almost every $\bi\in\cI^{\bN}$.
\begin{equation*}
	\lim_{n\to\infty}\frac{1}{n}\log|(f_{\bi}^n)'(x)|
	=\gamma(\mu).
\end{equation*}
Since for every $n\in\mathbb{N}$ and $\bi\in\cI^{\bN}$, $-n\log L\leq \log|(f_{\bi}^n)'(x)|\leq n\log L$, with $L$ as in \eqref{maxLip}. Hence, dominated convergence implies that \eqref{eq:lim-liap-exp-1} holds.

Finally, let us consider synchronization rates and prove
\begin{equation}\label{eq:lim-liap-exp-2}
    \lim_{n\to\infty}\frac{1}{n}\int\log d(X_n^x(\bi),X_n^y(\bi))\,d\mu^{\bN}(\bi)
	=\lim_{n\to\infty}\frac{1}{n}\int\log \frac{d(X_n^x(\bi),X_n^y(\bi)}{d(x,y)}\,d\mu^{\bN}(\bi)
	=\gamma(\mu).
\end{equation}
Indeed, by the mean value inequality, for every $x,y\in\bS^1$, $x\neq y$, and $\bi\in\Omega^{x,y}$
\[
 	\min_{z,w\in I_\bi}\left|\frac{(f_{\bi}^n)'(z)}{(f_{\bi}^n)'(w)}\right||(f_{\bi}^n)'(x)|
 	\leq \frac{d(X_n^x(\bi),X_n^y(\bi))}{d(x,y)}
 	\leq \max_{z,w\in I_\bi}\left|\frac{(f_{\bi}^n)'(z)}{(f_{\bi}^n)'(w)}\right||(f_{\bi}^n)'(x)|.
\]
Notice that
\[
	\min_{z,w\in I_\bi}\left|\frac{(f_{\bi}^n)'(z)}{(f_{\bi}^n)'(w)}\right|
	=\left( \max_{z,w\in I_\bi}\left|\frac{(f_{\bi}^n)'(z)}{(f_{\bi}^n)'(w)}\right|\right)^{-1}.
\]
Hence, by Lemma \ref{clac2dif}, for almost every $\bi\in\cI^{\bN}$
\[\begin{split}
	\lim_{n\to\infty}\frac{1}{n}\log d(X_n^x(\bi),X_n^y(\bi))
	&=\lim_{n\to\infty}\frac{1}{n}\log \frac{d(X_n^x(\bi),X_n^y(\bi))}{d(x,y)}\\
	&=\lim_{n\to\infty}\frac{1}{n}\log|(f_{\bi}^n)'(x)|
	=\gamma(\mu).
\end{split}\]
By dominated convergence, we can conclude that \eqref{eq:lim-liap-exp-2} holds.
This finishes the proof of the theorem.
\end{proof}

\bibliographystyle{alpha}

\end{document}